\documentclass[12pt,leqno]{amsart}
\usepackage{amsfonts,amsthm,amssymb,amsmath,amscd}
\usepackage{mathrsfs}

\def\N{{\mathbb N}}
\def\Z{{\mathbb Z}}
\def\Qq{{\mathbb Q}}
\def\R{{\mathbb R}}
\def\C{{\mathbb C}}
\def\H{{\mathfrak H}}

\def\sq{\hbox{\rlap{$\sqcap$}$\sqcup$}}
\def\qed{\ifmmode\sq\else{\unskip\nobreak\hfil
         \penalty50\hskip1em\null\nobreak\hfil\sq
         \parfillskip=0pt\finalhyphendemerits=0\endgraf}\fi}

\def\smat#1#2#3#4{\left(\begin{smallmatrix}#1&#2\\#3&#4\end{smallmatrix}\right)}

\def\M{{\mathcal M}}
\def\N{{\mathcal N}}

\newtheorem{theorem}{Theorem}
\newtheorem{lemma}[theorem]{Lemma}
\newtheorem{prop}[theorem]{Proposition}

\newtheorem{df}[theorem]{Definition}

\numberwithin{theorem}{section}
\numberwithin{equation}{section}

\allowdisplaybreaks[1]

\title[
A certain Dirichlet series of Rankin-Selberg type
]{
A certain Dirichlet series of Rankin-Selberg type
associated with the Ikeda lift of half-integral weight
}
\author[S.~Hayashida]{Shuichi Hayashida
}
\date{\today}

\begin{document}

\begin{abstract}
In this article we obtain an explicit formula for certain Rankin-Selberg  type
Dirichlet series associated to certain Siegel cusp forms of half-integral weight.
Here these Siegel cusp forms of half-integral weight are obtained
from the composition of the Ikeda lift and the Eichler-Zagier-Ibukiyama
correspondence.
The integral weight version of the main theorem had been obtained by
Katsurada and Kawamura.
The result of the integral weight case is a
product of $L$-function and Riemann zeta functions,
while half-integral weight case
is a infinite summation over negative fundamental discriminants
with certain infinite products.
To  calculate explicit formula of
such Rankin-Selberg type Dirichlet series,
we use a generalized Maass relation
and adjoint maps of index-shift maps of Jacobi forms.
\end{abstract}

\maketitle

\section{Introduction}\label{s:intro}

The purpose of this paper is to show a certain formula for Dirichlet series
of Rankin-Selberg type of certain Siegel cusp forms of \textit{half-integral weight}.

Let $2n$ and $k$ be positive even integers such that $k > 2n + 1$.
Let $g$ be a cusp form of weight $k-n+\frac12$ in the Kohnen plus-space.
Let $f$ be a normalized Hecke eigenform of elliptic cusp form of weight $2k-2n$
which corresponds to $g$ by the Shimura correspondence.
Let $F$ be the Siegel cusp form of weight $k$ of degree $2n$ which is the Ikeda lift
(the Duke-Imamo\={g}lu-Ikeda lift) of $g$.
In~\cite{KaKa:Dirichlet} Katsurada and Kawamura obtained the identity
\begin{eqnarray*}
  \zeta(2s - 2k + 4n) \sum_{N=1}^\infty \frac{\langle \psi_N, \psi_N \rangle}{N^s} &=& \langle \psi_1, \psi_1 \rangle \zeta(s - k + 1) \zeta(s - k + 2n) L(s,f),
\end{eqnarray*}
where $\psi_\N$ denotes the $N$-th Fourier-Jacobi coefficient of $F$ and
$\langle \psi_N, \psi_N \rangle$ denotes the Petersson inner product of $\psi_N$.
Here $\zeta(s) $ denotes the Riemann zeta function and $L(s,f)$ denotes the usual $L$-function
of $f$.
In the case of $n=1$ the above formula had been obtained by Kohnen and Skoruppa in \cite{KoSk}.

Let $G$ be the Siegel cusp form of weight $k-\frac12$ of degree $2n-1$ which is obtained
from the $1$st Fourier-Jacobi coefficient $\psi_1$ by the Eichler-Zagier-Ibukiyama correspondence.
Here the Eichler-Zagier-Ibukiyama correspondence is the linear isomorphism
between Jacobi forms of index 1 and the generalized plus-space of Siegel modular forms
(see~\cite{Ib}).
The form $G$ belongs to the generalized plus-space.
We remark that if $n=1$, then $G = g$.
Let $\phi_m$ be the $m$-th Fourier-Jacobi coefficient of $G$ for any natural number $m$.
(See \S\ref{s:maass_relation}).
We remark that $\phi_m$ is a Jacobi cusp form of weight $k-\frac12$ of index $m$ of degree $2n-2$
and $\phi_m$ is identically $0$ unless $-m \equiv 0, 1 \!\! \mod 4$.
We denote by $\langle \phi_m, \phi_m \rangle$ the Petersson inner product  of $\phi_m$.
(See \S\ref{ss:def_jacobi_half_weight} for the definition).
If $n=1$, then $\phi_m$ is the $m$-th Fourier coefficient of $g$ and we put
$\langle \phi_m, \phi_m \rangle := |\phi_m|^2$ in this case.

The aim of this paper is to show the following theorem.
\begin{theorem}\label{th:main_dshalf}
Let $f$ be a normalized Hecke eigenform of elliptic cusp form of weight $2k-2n$.
Let the symbols be as above. We have
 \begin{eqnarray*}
    &&
    \hspace{-0.5in}
     \sum_{\begin{smallmatrix} m \in \Z_{>0} \\ -m \equiv 0,1 \!\! \mod 4 \end{smallmatrix}}
      \frac{\langle  \phi_m , \phi_m  \rangle}{m^{s + k - n - \frac12}} \\
   &=&
     \zeta(2s - 2n + 2)
    \, \zeta(4s)^{-1}
    \, L(2s, f, Ad) \\
    &&
    \times \sum_{D_0}  \frac{\langle \phi_{|D_0|}, \phi_{|D_0|} \rangle }{{|D_0|}^{s+k-n-\frac12}}
         \prod_{p \nmid |D_0|}
    \left\{
       1 + p^{-2s - 1}
     - \frac{\left( \frac{D_0}{p} \right)
     \left( 1 + p^{2n-2} \right)
     p^{-k+1} a_f(p) }{
      1 + p^{2s}
     }
     \right\}
 \end{eqnarray*}
 for sufficient large $\mbox{Re}(s)$,
 where $D_0 < 0$ runs over all fundamental discriminants
 and $a_f(p)$ denotes the $p$-th Fourier coefficient of $f$.
 Here $L(s,f,Ad)$ denotes the adjoint L-function of $f$:
 \begin{eqnarray*}
   L(s,f,Ad)
   &:=&
   \prod_p 
   \left\{ \left(1 - p^{-s} \right) \left( 1 - \alpha_p^2 p^{-s}\right) \left( 1 - \alpha_p^{-2} p^{-s}\right) \right\}^{-1},
 \end{eqnarray*}
 where $\left\{ \alpha_p^{\pm} \right\}$ are complex values determined by
 the identity 
 \begin{eqnarray*}
   a_f(p) &=& \left(\alpha_p + \alpha_p^{-1} \right) p^{k-n-\frac12}.
 \end{eqnarray*}
\end{theorem}
In the case of $n=1$ the above formula coincides with the formula 
in~\cite[p.182, l.5]{KoZa}
with a modification that the function
$\zeta(2s) \zeta(4s)^{-1} L(2s,f,Ad)$ should be multiplied in the left hand side in
~\cite[p.182, l.5]{KoZa}.

To obtain main theorem,
a generalization of the Maass relation (Proposition~\ref{prop:maass_relation_D})
plays an important rule. 
This generalization of the Maass relation had been shown essentially in~\cite[Theorem 8.2]{CE_N}.
We also need calculations of adjoint maps of index-shift maps.

We remark that we have analytic properties of the above Dirichlet series
by using Rankin-Selberg method for generalized plus-space which is shown in~\cite[Corollary 3.4]{RSJacobi}.
We put
\begin{eqnarray*}
  \mathcal{R}_1(G;s)
  &:=&
  \pi^{-2s} \Gamma(s+k-n-\frac12) \Gamma(s+n-1) \zeta(2s+2n-2) \\
  && \times
       \sum_{\begin{smallmatrix} m \in \Z_{>0} \\ -m \equiv 0,1 \!\! \mod 4 \end{smallmatrix}}
      \frac{\langle  \phi_m , \phi_m  \rangle}{m^{s+k-n-\frac12}}.
\end{eqnarray*}
Then $\mathcal{R}_1(G;s)$ has meromorphic continuation to the whole complex plane
and satisfies the functional equation
\begin{eqnarray*}
  \mathcal{R}_1(G;s) &=&  \mathcal{R}_1(G;1-s).
\end{eqnarray*}
The function $\mathcal{R}_1(G;s)$ is entire except for $s = n$ and $s = 1-n$.
The residue at $s = n$ is
\begin{eqnarray*}
  \mbox{Res}_{s = n} \mathcal{R}_1(G;s)
  &=& (1 + \delta_{n,1})^{-1} 2^{2k-1}\pi^{k-\frac32} \langle G, G \rangle,
\end{eqnarray*}
where we put $\delta_{n,1} := 1$ if $n=1$ and $\delta_{n,1} := 0$ if $n > 1$,
and $\langle G, G \rangle$ denotes the Petersson inner product of $G$.
Moreover, an explicit formula for $\langle G, G \rangle$ is shown in~\cite{KaKa}.
To describe the value $\langle G, G \rangle$ we prepare some symbols.

We put $\Gamma_\C(s) := 2(2\pi)^{-s} \Gamma(s)$
and put $\tilde{\xi}(s) := \Gamma_\C(s) \zeta(s)$.
We set
\begin{eqnarray*}
  \tilde{\Lambda}(s,f,Ad)
  &:=&
  \Gamma_\C(s) \Gamma_\C(s+2k-2n-1) L(s,f,Ad).
\end{eqnarray*}
Then it is known in~\cite{KaKa} that 
\begin{eqnarray*}
  \langle G, G \rangle
  &=& (1 + \delta_{n,1}) 2^{-6k(n-1)+n(2n-3)} \langle g, g \rangle
  \prod_{i=1}^{n-1} \tilde{\xi}(2i) \tilde{\Lambda}(2i+1,f,Ad).
\end{eqnarray*}
Therefore the residue of $\mathcal{R}_1(G;s)$ at $s = n$ is
\begin{eqnarray*}
  \mbox{Res}_{s=n} \mathcal{R}_1(G;s)
  &=&
  2^{-2k(3n-4)+n(2n-3)-1} \pi^{k-\frac32} \langle g, g \rangle \prod_{i=1}^{n-1} \tilde{\xi}(2i) \tilde{\Lambda}(2i+1,f,Ad).
\end{eqnarray*}

Remark also 
that the infinite product
\begin{eqnarray*}
        \prod_{p \nmid |D_0|}
    \left\{
       1 + p^{-2s - 1}
     - \frac{\left( \frac{D_0}{p} \right)
     \left( 1 + p^{2n-2} \right)
     p^{-k+1} a_f(p) }{
      1 + p^{2s}
     }
     \right\}
\end{eqnarray*}
in Theorem~\ref{th:main_dshalf}
appears in a formula of
a certain two variable Dirichlet series $\mathscr{L}_{-1}(f;\lambda,s)$ associated to $f$ (cf. \cite[p.225]{IbKa}).

\ \\

This article is organized as follows.
In \S\ref{s:notation} we prepare some symbols. We also recall definitions of Jacobi forms
and the index shift maps of Jacobi forms.
In \S\ref{s:Ikeda_lift} we recall Ikeda lifts and construct Siegel modular forms of
half-integral weight.
In \S\ref{s:maass_relation} we review a generalization of the Maass relation for
Siegel modular forms of half-integral weight. We also introduce a index shift map
$D_{2n-2}(N)$ for Jacobi forms of half-integral weight.
In \S\ref{s:lin_iso_jacobi} we review a linear isomorphism between
Jacobi forms of integral weight and half-integral weight.
In \S\ref{s:adjoint_maps} we give a formula for the adjoint map $D^*_{n'}(N)$
of $D_{n'}(N)$ with respect to the Petersson inner product.
In \S\ref{s:Jacobi_Eisenstein} we review the Fourier-Jacobi coefficients
of the generalized Cohen-Eisenstein series and Jacobi-Eisenstein series.
We also calculate the image of them by an adjoint map $U^*_N$ of a certain index-shift map
$U_N$.
Finally, in \S\ref{s:proof_of_main_theorem} we will give a proof of Theorem~\ref{th:main_dshalf}.

\ \\

\noindent
Acknowledgement:

The author would like to express his sincere thanks to Professor Hidenori Katsurada
for his variable comments.
This work was supported by JSPS KAKENHI Grant Number 80597766.

\section{Notation and definitions}\label{s:notation}
We denote by $\Z_{>0}$ (resp. $\R_{>0}$) the set of all positive integers
(resp. positive real numbers).
The symbol 
$R^{(n,m)}$ denotes the set of $n\times m$ matrices with entries in a ring $R$.
The symbol $L_n^*$ denotes the set of all semi positive-definite, half-integral symmetric matrices
of size $n$,
and the symbol
$L_n^+$ denotes the set of all positive definite, half-integral symmetric matrices of size $n$.
The transpose of a matrix $B$ is denoted by $^t B$.
We write $A[B]$ $:=$ ${^t B} A B $ for two matrices $A \in R^{(n,n)}$ and $B \in R^{(n,m)}$ .
We write the identity matrix (resp. zero matrix) of size $n$ by $1_n$ (resp. $0_n$).
We denote by $\mbox{tr}(S)$ the trace of a square matrix $S$
and we write $e(S) := e^{2 \pi \sqrt{-1}\, \mbox{tr}(S)}$ for a square matrix $S$.
For square matrices $a_1$, ..., $a_n$, we denote by
$\mbox{diag}(a_1,...,a_n)$
 the diagonal matrix
$\left(\begin{smallmatrix} a_1 & & \\ &\ddots & \\ & & a_n\end{smallmatrix}\right)$.
The symbol $p$ is reserved for prime number.
For any odd prime $p$ the symbol $\left(\frac{*}{p}\right)$ denotes the Legendre symbol.
If $p = 2$, then we denote by $\left(\frac{d}{2}\right) = 1$, $-1$ or $0$
for $d \equiv \pm 1 \!\! \mod 8$, $d \equiv \pm 3 \!\! \mod 8$ or $d \equiv 0 \mod 2$,
respectively.

We denote by $\mathfrak{H}_n$ the Siegel upper half space of degree $n$
and  denote by $\mbox{Sp}_n(\R)$ the real symplectic group of size $2n$.
We set $\Gamma_n$ := $\mbox{Sp}_n(\Z)$.
We put $\Gamma_0^{(n)}(4) := \left. \left\{ \begin{pmatrix} A & B \\ C & D \end{pmatrix} \in \Gamma_n
\, \right| \, C \in 4 \Z^{(n,n)} \right\}$
and put $ \Gamma_{\infty}^{(n)}
 :=
 \left.
 \left\{
  \begin{pmatrix} A & B \\ C & D \end{pmatrix} \in \Gamma_n
 \, \right| \,
  C = 0_n 
 \right\}
$.
We denote by
$M_{k-\frac12}^{(n)}$  the vector space of Siegel modular forms of weight $k-\frac12$ of
$\Gamma_0^{(n)}(4)$.
We put $M_{k-\frac12}^{+(n)}$
the plus-space of weight $k-\frac12$ of degree $n$,
which is a certain subspace of $M_{k-\frac12}^{(n)}$
and it is a generalization of Kohnen plus-space for general degree (cf. Ibukiyama~\cite{Ib},
see also \S\ref{s:lin_iso_jacobi}).
The symbol
$S_{k-\frac12}^{+(n)}$ denotes the vector space of all Siegel cusp forms in $M_{k-\frac12}^{+(n)}$.

In the following we quote some symbols and definitions from~\cite{CE_N}.
\subsection{Jacobi group}\label{ss:Jacobi_group}
For a positive integer $n$ we define the group
\begin{eqnarray*}
 \mbox{GSp}_n^+(\R)
 \!\!\!\!\!
 &:=& \!\!\!\!\!
 \left\{
  g \in \R^{(2n,2n)} \, 
  | \, 
  g \left(\begin{smallmatrix}
     0_n & -1_n \\ 1_n & 0_n 
   \end{smallmatrix}\right)
  {^t g} 
  = n(g) 
  \left(\begin{smallmatrix}
     0_n & -1_n \\ 1_n & 0_n 
   \end{smallmatrix}\right)
  \mbox{ for some } n(g) \in \R_{>0}
 \right\}.
\end{eqnarray*}
For a matrix $g \in \mbox{GSp}_n^+(\R)$, the number $n(g)$ in the above definition of $\mbox{GSp}_n^+(\R)$ is called the \textit{similitude} of the matrix $g$.

For positive integers $n$ and $r$,
we define a subgroup $G_{n,r}^J \subset \mbox{GSp}_{n+r}^+(\R)$ by
\begin{eqnarray*}
 G_{n,r}^J 
 &:=&
 \left\{
 \left.
  \begin{pmatrix}
   A &   & B &  \\
     & U &   &  \\
   C &   & D &  \\
     &   &   & V
\end{pmatrix}\begin{pmatrix}
   1_n &   &   & \mu  \\
   ^t \lambda  & 1_r & ^t \mu  & {^t \lambda} \mu + \kappa  \\
     & & 1_n & - \lambda  \\
     &   &   & 1_r
\end{pmatrix} 
\in \mbox{GSp}_{n+r}^+(\R)
 \right|
 \begin{matrix}
 A, B, C, D, \\
 U, V, \\
 \lambda, \mu, \kappa
 \end{matrix}
\right\},
\end{eqnarray*}
where
$\begin{pmatrix}
  A&B\\C&D
 \end{pmatrix}
 \in \mbox{GSp}_n^+(\R)$,
$\begin{pmatrix}
  U&0\\0&V
 \end{pmatrix}
 \in \mbox{GSp}_r^+(\R)$,
$\lambda, \mu \in \R^{(n,r)}$
and
$\kappa = {^t \kappa}\in \R^{(r,r)}$.
We remark that two matrices $\left(\begin{smallmatrix}A&B\\C&D\end{smallmatrix}\right)$ and
$\left(\begin{smallmatrix}U&0\\0&V\end{smallmatrix}\right)$
in the above notation
have the same similitude.
We abbreviate an element
 $\left(\begin{smallmatrix}
    A &   & B &  \\
      & U &   &  \\
    C &   & D &  \\
      &   &   & V
 \end{smallmatrix}\right)
 \left(\begin{smallmatrix}
   1_n &   &   & \mu  \\
   ^t \lambda  & 1_r & ^t \mu  & {^t \lambda}\mu + \kappa  \\
     & & 1_n & - \lambda  \\
     &   &   & 1_r
 \end{smallmatrix}\right)$
as
\[
\left(\begin{pmatrix}A&B\\C&D\end{pmatrix}
  \times
  \begin{pmatrix}U&0\\0&V\end{pmatrix},
  [
   (\lambda,\mu),\kappa
  ]
  \right) .
\]
We will often write
\begin{eqnarray*}
\left(\left(\begin{matrix}A&B\\C&D\end{matrix}\right),
  [
   (\lambda,\mu),\kappa
  ]
  \right)
\end{eqnarray*}
instead of
 $\left(\left(\begin{smallmatrix}A&B\\C&D\end{smallmatrix}\right)
  \times
  1_{2r},
  [
   (\lambda,\mu),\kappa
  ]
  \right)$
for simplicity.
The element
 $\left(\left(\begin{smallmatrix}A&B\\C&D\end{smallmatrix}\right),
  [
   (\lambda,\mu),\kappa
  ]
  \right)$
belongs to $\mbox{Sp}_{n+r}(\R) $.
Similarly, we abbreviate an element
\begin{eqnarray*}
 \left(\begin{smallmatrix}
   1_n &   &   & \mu  \\
   ^t \lambda  & 1_r & ^t \mu  & {^t \lambda}\mu + \kappa  \\
     & & 1_n & - \lambda  \\
     &   &   & 1_r
 \end{smallmatrix}\right)
 \left(\begin{smallmatrix}
    A &   & B &  \\
      & U &   &  \\
    C &   & D &  \\
      &   &   & V
 \end{smallmatrix}\right)
\end{eqnarray*}
as
\begin{eqnarray*}
 \left(
  [
   (\lambda,\mu),\kappa
  ],
  \left(\begin{matrix}A&B\\C&D\end{matrix}\right)
  \times
  \left(\begin{matrix}U&0\\0&V\end{matrix}\right)
 \right),
\end{eqnarray*}
and abbreviate it as
 $
 \left(
  [
   (\lambda,\mu),\kappa
  ],
  \left(\begin{smallmatrix}A&B\\C&D\end{smallmatrix}\right)
 \right)
 $
for the case $U = V = 1_r$ .

If there is no confusion, we write 
\[
[(\lambda,\mu),\kappa]
\]
for the element $(1_{2n},[(\lambda,\mu),\kappa])$ for simplicity.

We set a subgroup  of $G_{n,r}^J$ by
\begin{eqnarray*}
 \Gamma_{n,r}^J 
  &:=&
 \left\{
  \left(M,
  [
   (\lambda,\mu),\kappa
  ]
  \right) 
   \in G_{n,r}^J
  \, \left| \,
  M \in \Gamma_n ,
  \lambda, \mu \in \Z^{(n,r)}, \kappa \in \Z^{(r,r)}
 \right\} \right. .
\end{eqnarray*}

\subsection{Groups $\widetilde{\mbox{GSp}_n^+(\R)}$
and $\widetilde{G_{n,1}^J}$}\label{ss:double_covering_groups}

The symbol $\widetilde{\mbox{GSp}_n^+(\R)}$
denotes the group
which consists of pairs $(M,\varphi(\tau))$,
where $M$ is a matrix $M = \left(\begin{smallmatrix} A&B\\C&D \end{smallmatrix}\right)\in \mbox{GSp}_n^+(\R)$,
and where $\varphi$ is a holomorphic function on $\mathfrak{H}_n$
which satisfies $|\varphi(\tau)|^2 = \det(M)^{-\frac12} |\det(C\tau + D)|$.
The group operation on $\widetilde{\mbox{GSp}_n^+(\R)}$ is given by
$(M,\varphi(\tau))(M',\varphi'(\tau)) := (M M', \varphi(M'\tau)\varphi'(\tau))$
for $(M,\varphi)$, $(M',\varphi')$ $\in$ $\widetilde{\mbox{GSp}_n^+(\R)}$.

We denote the theta constant  $ \theta^{(n)}(\tau) 
  :=
   \displaystyle{\sum_{p \in \Z^{(n,1)}} e(\tau[p])}$
for $\tau \in \H_n$.
We embed $\Gamma_0^{(n)}(4)$ into the group $\widetilde{\mbox{GSp}_n^+(\R)}$
via $M \mapsto (M,\theta^{(n)}(M\tau)\, \theta^{(n)}(\tau)^{-1})$.

We denote by $\Gamma_0^{(n)}(4)^*$ the image of $\Gamma_0^{(n)}(4)$
in $\widetilde{\mbox{GSp}_n^+(\R)}$ by this embedding.

We define the group
\begin{eqnarray*}
 H_{n,1}(\R)
 &:=&
 \left\{[(\lambda,\mu),\kappa] \in \mbox{Sp}_{n+1}(\R)
 \, | \,
 \lambda,\mu \in \R^{(n,1)}, \kappa \in \R \right\} .
\end{eqnarray*}
and define the group
\begin{eqnarray*}
 \widetilde{G_{n,1}^J}
 &:=&
 \widetilde{\mbox{GSp}_n^+(\R)} \ltimes H_{n,1}(\R) \\
 &=&
 \left. \left\{(\tilde{M},[(\lambda,\mu),\kappa]) \, \right| \,
   \tilde{M} \in \widetilde{\mbox{GSp}_n^+(\R)},
   [(\lambda,\mu),\kappa] \in H_{n,1}(\R) \right\} 
\end{eqnarray*}
with the group operation
\begin{eqnarray*}
 (\tilde{M_1},[(\lambda_1,\mu_1),\kappa_1])\cdot (\tilde{M_2},[(\lambda_2,\mu_2),\kappa_2])
 &:=&
 (\tilde{M_1}\tilde{M_2},[(\lambda',\mu'),\kappa'])
\end{eqnarray*}
for $(\tilde{M_i},[(\lambda_i,\mu_i),\kappa_i]) \in \widetilde{G_{n,1}^J}$ $(i=1,2)$,
and where $[(\lambda',\mu'),\kappa'] \in H_{n,1}(\R)$ is the matrix determined through the identity
\begin{eqnarray*}
\begin{aligned}
 &
 (M_1\times\left(\begin{smallmatrix}n(M_1)&0\\0&1\end{smallmatrix}\right),[(\lambda_1,\mu_1),\kappa_1])
 (M_2\times\left(\begin{smallmatrix}n(M_2)&0\\0&1\end{smallmatrix}\right),[(\lambda_2,\mu_2),\kappa_2])
\\
 &=
 (M_1M_2 \times \left(\begin{smallmatrix}n(M_1)n(M_2)&0\\0&1\end{smallmatrix}\right),[(\lambda',\mu'),\kappa'])
\end{aligned}
\end{eqnarray*}
in $G_{n,1}^J$.
Here $n(M_i)$ is the similitude of $M_i$.

\subsection{Action of the Jacobi group}
The group $G_{n,r}^J$ acts on $\mathfrak{H}_n \times \C^{(n,r)}$ by
\begin{eqnarray*}
 \gamma \cdot (\tau,z) 
 &:=&
 \left(
  \begin{pmatrix}A&B\\C&D\end{pmatrix} \cdot \tau
 \, ,\,
  ^t(C\tau+D)^{-1}(z + \tau\lambda + \mu)^t U
 \right)
\end{eqnarray*}
for any $\gamma = \left(\left(\begin{smallmatrix}A&B\\C&D\end{smallmatrix}\right)
  \times
  \left(\begin{smallmatrix}U&0\\0&V\end{smallmatrix}\right),
  [
   (\lambda,\mu),\kappa
  ]
  \right) \in G_{n,r}^J$ 
and for any $(\tau,z) \in \mathfrak{H}_n \times \C^{(n,r)}$. 
Here $\begin{pmatrix}A&B\\C&D\end{pmatrix} \cdot \tau := (A\tau+B)(C\tau+D)^{-1}$ 
is the usual transformation.

The group $\widetilde{G_{n,1}^J}$ acts on $\mathfrak{H}_n\times \C^{(n,1)}$
through the projection $\widetilde{G_{n,1}^J} \rightarrow G_{n,1}^J$.
It means $\widetilde{G_{n,1}^J}$ acts on $\mathfrak{H}_n\times \C^{(n,1)}$
by
\begin{eqnarray*}
 \tilde{\gamma}\cdot(\tau,z) 
  &:=& 
  (M\times\left(\begin{smallmatrix}n(M)&0\\0&1\end{smallmatrix}\right),[(\lambda,\mu),\kappa])\cdot(\tau,z)
\end{eqnarray*}
for $\tilde{\gamma} = ((M,\varphi),[(\lambda,\mu),\kappa]) \in \widetilde{G_{n,1}^J}$
and for $(\tau,z) \in \mathfrak{H}_n\times \C^{(n,1)}$.
Here $n(M)$ is the similitude of $M \in \mbox{GSp}_n^+(\R)$.

\subsection{Factors of automorphy}\label{ss:factors_automorphy}
Let $k$ be an integer and let $\mathcal{M}$ be a symmetric matrix of size $r$ with entries in $\R$.
For any
 $\gamma = \left(\left(\begin{smallmatrix}A&B\\C&D\end{smallmatrix}\right)
  \times
  \left(\begin{smallmatrix}U&0\\0&V\end{smallmatrix}\right),
  [
   (\lambda,\mu),\kappa
  ]
  \right) \in G_{n,r}^J$
we define the factor of automorphy
\begin{eqnarray*}
 &&
 \hspace{-1cm}
 J_{k,\mathcal{M}}\left( 
   \gamma, (\tau,z) \right) \\
 &:=& 
  \det(V)^k \det(C\tau+D)^k\, e(V^{-1}\mathcal{M}U (((C\tau+D)^{-1}C)[z + \tau \lambda + \mu]) ) \\
 && \times
  e(- V^{-1} \mathcal{M} U ({^t \lambda} \tau \lambda
     + {^t z} \lambda + {^t \lambda} z + {^t \mu} \lambda + {^t \lambda}\mu + \kappa)).
\end{eqnarray*}
We define the slash operator $|_{k,\mathcal{M}}$ by
\begin{eqnarray*}
 (\psi|_{k,\mathcal{M}}\gamma)(\tau,z) 
 &:=& 
 J_{k,\mathcal{M}}(\gamma,(\tau,z))^{-1} \psi(\gamma\cdot(\tau,z))
\end{eqnarray*}
for any function $\psi$ on $\mathfrak{H}_n \times \C^{(n,r)}$ 
and for any $\gamma \in G_{n,r}^J$.
We remark that
\begin{eqnarray*}
 J_{k,\mathcal{M}}(\gamma_1 \gamma_2, (\tau,z))
 &=&
 J_{k,\mathcal{M}}(\gamma_1, \gamma_2 \cdot (\tau,z))
 J_{k,V_1^{-1}\mathcal{M}U_1}(\gamma_2, (\tau,z)),
\\
 \psi|_{k,\mathcal{M}}\gamma_1 \gamma_2 
 &=&
 (\psi|_{k,\mathcal{M}}\gamma_1) |_{k,V_1^{-1}\mathcal{M}U_1}\gamma_2 .
\end{eqnarray*}
for any $\gamma_i = \left(M_i
  \times
  \left(\begin{smallmatrix}U_i&0\\0&V_i\end{smallmatrix}\right),
  [
   (\lambda_i,\mu_i),\kappa_i
  ]
  \right) \in G_{n,r}^J$ $(i =1,2)$.

Let $k$ and $m$ be integers.
For any $\tilde{\gamma} = ((M,\varphi),[(\lambda,\mu),\kappa]) \in \widetilde{G_{n,1}^J}$
we define the factor of automorphy
\begin{eqnarray*} 
 J_{k-\frac12,m}(\tilde{\gamma},(\tau,z))
 &:=&
 \varphi(\tau)^{2k-1} e(n(M) m (((C\tau+D)^{-1}C)[z + \tau \lambda + \mu]) ) \\
 && \times
  e(-  n(M) m ({^t \lambda} \tau \lambda + {^t z} \lambda + {^t \lambda} z + {^t \mu} \lambda + {^t \lambda}\mu + \kappa)),
\end{eqnarray*}
where $n(M)$ is the similitude of $M$.
We define the slash operator $|_{k-\frac12,m}$ by
\begin{eqnarray*}
 \phi|_{k-\frac12,m}\tilde{\gamma}
 &:=&
 J_{k-\frac12,m}(\tilde{\gamma},(\tau,z))^{-1}
 \phi(\tilde{\gamma}\cdot(\tau,z))
\end{eqnarray*}
for any function $\phi$ on $\mathfrak{H}_n\times \C^{(n,1)}$.
We remark that
\begin{eqnarray*}
 J_{k-\frac12,m}(\tilde{\gamma_1}\tilde{\gamma_2},(\tau,z))
 &=&
 J_{k-\frac12,m}(\tilde{\gamma_1},\tilde{\gamma_2}\cdot(\tau,z))
 J_{k-\frac12,n(M_1) m}(\tilde{\gamma_2},(\tau,z))
\\
 \phi|_{k-\frac12,m}\tilde{\gamma_1}\tilde{\gamma_2}
 &=&
 (\phi|_{k-\frac12,m}\tilde{\gamma_1})|_{k-\frac12,n(M_1)m}\tilde{\gamma_2}
\end{eqnarray*}
for any $\tilde{\gamma_i} = ((M_i,\varphi_i),[(\lambda_i,\mu_i),\kappa_i]) \in \widetilde{G_{n,1}^J}$
$(i=1,2)$.

\subsection{Jacobi forms of matrix index}\label{ss:jacobi_forms_of_matrix_index}
We quote the definition of Jacobi forms of matrix index from \cite{Zi}.
\begin{df}
For an integer $k$ and for an matrix $\mathcal{M} \in L_r^+$,
a $\C$-valued holomorphic function $\psi$ on $\mathfrak{H}_n \times \C^{(n,r)}$ is called
\textit{a Jacobi form of weight $k$ of
index $\mathcal{M}$ of degree $n$}, if $\psi$ satisfies the following two conditions:
\begin{enumerate}
\item
the transformation formula
$\psi|_{k,\mathcal{M}} \gamma = \psi$ for any $\gamma \in \Gamma_{n,r}^J$,
\item
$\psi$ has the Fourier expansion:
$ \psi(\tau,z)
 = \!\!\!\!\!
 \displaystyle{
   \sum_{\begin{smallmatrix}N \in Sym_n^*,R \in Z^{(n,r)} \\ 4N - R \mathcal{M}^{-1} {^t R} \geq 0 
          \end{smallmatrix}} \!\!\!\!\! c(N,R) e(N\tau) e({^t R} z)
              }$.
\end{enumerate}
\end{df}
We remark that the second condition follows from the Koecher principle (cf.~\cite[Lemma~1.6]{Zi})
if $n > 1$.
In the condition (2),
if $\psi$ satisfies $c(N,R) = 0$ unless $4N - R \mathcal{M}^{-1} {^t R} > 0$,
then $\psi$ is called a \textit{Jacobi cusp form}.

We denote by $J_{k,\mathcal{M}}^{(n)}$ (resp. $J_{k,\mathcal{M}}^{(n)\, cusp}$)
the $\C$-vector space of Jacobi forms (resp. Jacobi cusp forms) of weight $k$ of index $\mathcal{M}$
of degree $n$.

For $\psi_1$, $\psi_2$ $\in$ $J_{k,\M}^{(n)\, cusp}$,
the Petersson inner product is defined by
\begin{eqnarray*}
  \langle \psi_1, \psi_2 \rangle
  &:=&
  \int_{\mathcal{F}_{n,r}}
    \psi_1(\tau,z) \overline{ \psi_2(\tau,z)} e^{-4\pi Tr( \M v^{-1}[y])} \det(v)^{k - n - r - 1 } 
    \, du\, dv\, dx\, dy,
\end{eqnarray*}
where $\mathcal{F}_{n,r} := \Gamma_{n,r}^J \backslash (\H_n \times \C^{(n,r)}) $,
$\tau = u + i v$, $z = x + i y$,
$du = \prod_{i \leq j} u_{i,j}$,  $d v = \prod_{i \leq j} v_{i,j}$,
$dx = \prod_{i,j} x_{i,j}$ and $dy = \prod_{i,j} y_{i,j}$.

\subsection{Jacobi forms of half-integral weight}\label{ss:def_jacobi_half_weight}

We set a subgroup $\Gamma_{n,1}^{J*}$ of $\widetilde{G_{n,1}^J}$ by
\begin{eqnarray*}
 \Gamma_{n,1}^{J*}
 &:=&
 \left\{
  (M^*,[(\lambda,\mu),\kappa]) \in \widetilde{G_{n,1}^J}
  \, | \,
  M^* \in \Gamma_0^{(n)}(4)^*, \,
  \lambda,\mu \in \Z^{(n,1)}, \kappa \in \Z
 \right\} \\
 &\cong&
  \Gamma_0^{(n)}(4)^* \ltimes H_{n,1}(\Z),
\end{eqnarray*}
where we put $H_{n,1}(\Z) := H_{n,1}(\R) \cap \Z^{(2n+2,2n+2)}$,
and where the group $\Gamma_0^{(n)}(4)^*$ was defined in \S\ref{ss:double_covering_groups}.
\begin{df}
For integers $k$ and $m$,
a holomorphic function $\phi$ on $\mathfrak{H}_n \times \C^{(n,1)}$
is called a \textit{Jacobi form of weight $k-\frac12$ of index $m$ of degree $n$},
if $\phi$ satisfies the following two conditions:
\begin{enumerate}
\item
$\phi|_{k-\frac12,m} \gamma^* = \phi$ for any $\gamma^* \in \Gamma_{n,1}^{J*}$,
\item
$\phi^2|_{2k-1,2m}\gamma$ has the Fourier expansion for any $\gamma \in \Gamma_{n,1}^J$:
\begin{eqnarray*}
 \left(\phi^2|_{2k-1,2m}\gamma\right) (\tau,z)
 &=&
 \sum_{\begin{smallmatrix}
        N \in Sym_n^*,R \in \Z^{(n,1)} \\
        4Nm - h R {^t R} \geq 0
       \end{smallmatrix}}
      C(N,R)\, e\!\left(\frac{1}{h}N\tau\right) e\!\left({^t R} z\right).
\end{eqnarray*}
with a  integer $h > 0$, and
where the slash operator $|_{2k-1, 2m}$ was defined in $\S\ref{ss:factors_automorphy}$.
\end{enumerate}
\end{df}

In the condition (2),
for any $\gamma$ if $\phi$ satisfies $C(N,R) = 0$ unless $4Nm - h R{^t R} > 0$,
then $\phi$ is called a \textit{Jacobi cusp form}.

We denote by $J_{k-\frac12,m}^{(n)}$
the $\C$-vector space of Jacobi forms
of weight $k-\frac12$ of index $m$ of degree $n$.

For $\phi_1$, $\phi_2$ $\in$ $J_{k-\frac12,m}^{(n)\, cusp}$, 
the Petersson inner product is defined by
\begin{eqnarray*}
  \langle \phi_1, \phi_2 \rangle
  &:=&
   \left[ \Gamma_n : \Gamma_0^{(n)}(4) \right]^{-1}
  \int_{\mathcal{F}_{n,1,4}}
    \phi_1(\tau,z) \overline{ \phi_2(\tau,z)} e^{-4\pi m v^{-1}[y]} \det(v)^{k - n - \frac52 } 
    \, du\, dv\, dx\, dy,
\end{eqnarray*}
where $\mathcal{F}_{n,1,4} := \Gamma_{n,1}^J(4) \backslash (\H_n \times \C^{(n,1)}) $,
$\tau = u + i v$, $z = x + i y$,
$du = \prod_{i \leq j} u_{i,j}$,  $d v = \prod_{i \leq j} v_{i,j}$,
$dx = \prod_{i} x_{i,1}$ and $dy = \prod_{i} y_{i,1}$
and $\left[ \Gamma_n : \Gamma_0^{(n)}(4) \right]$ denotes the index of
$\Gamma_0^{(n)}(4)$ in $\Gamma_n$,
and where we put
\begin{eqnarray*}
  \Gamma_{n,1}^J(4)
  &:=&
 \left\{
  (M,[(\lambda,\mu),\kappa]) \in \Gamma_{n,1}^J
  \, | \,
  M \in \Gamma_0^{(n)}(4), \,
  \lambda,\mu \in \Z^{(n,1)}, \kappa \in \Z
 \right\} \\
 &\cong&
  \Gamma_0^{(n)}(4) \ltimes H_{n,1}(\Z).
\end{eqnarray*}

\begin{lemma}\label{lem:cusp_criterium}
 Let $\phi \in J_{k-\frac12,m}^{(n)}$.
 Then $\phi$ is a Jacobi cusp form, if and only if the function
 \begin{eqnarray*}
   \det(v)^{\frac12(k-\frac12)}
   e^{-2\pi m v^{-1}[y] }
   |\phi(\tau,z)|
 \end{eqnarray*}
 is bounded on $\H_{n} \times \C^{n}$.
 Here we put $v = \mbox{Im}(\tau)$ and $y = \mbox{Im}(z)$.
\end{lemma}
\begin{proof}
 It is an analogue to~\cite[p.410 Lemma]{Kl2} and ~\cite[Corollary to Proposition 1]{Du}.
 Here we omitted the detail.
\end{proof}

\subsection{Index-shift maps of Jacobi forms}\label{ss:hecke_operators}
In this subsection we introduce index-shift maps for two kinds of Jacobi forms
(of matrix index and of half-integral weight).
These are generalizations of
the $V_l$-map in the sense of Eichler-Zagier \cite{EZ}.
 
We define $\mbox{GSp}_n^+(\Z) := \mbox{GSp}_n^+(\R) \cap \Z^{(2n,2n)}$ and
\begin{eqnarray*}
 \widetilde{\mbox{GSp}_n^+(\Z)}
 &:=&
 \left.
 \left\{
  (M,\varphi) \in \widetilde{\mbox{GSp}_n^+(\R)} \, \right| \, M \in \mbox{GSp}_n^+(\Z)
 \right\}.
\end{eqnarray*}

First we define index-shift maps for Jacobi forms of integral weight with some matrix indices.

Let $\mathcal{M} = \smat{*}{*}{ * }{1} \in L_2^+$.
We take a matrix $X \in \mbox{GSp}_n^+(\Z)$ such that the similitude of $X$ is $n(X)=N^2$
with a natural number $N$.
For any $\psi \in J_{k,\mathcal{M}}^{(n)}$ we define the function
\begin{eqnarray*}
 &&
 \hspace{-1cm}
 \psi|V(X) \\
 &:=&
 \sum_{u,v \in (\Z/N\Z)^{(n,1)}}
 \sum_{M \in \Gamma_n \backslash \Gamma_n X \Gamma_n} \!\!\!\!\!
  \psi|_{k,\mathcal{M}} 
   \left(M\times
     \left(\begin{smallmatrix}N^2&0&0&0\\0&N&0&0\\0&0&1&0\\0&0&0&N\end{smallmatrix}\right),
     [((0,u),(0,v)),0_2]\right),
\end{eqnarray*}
where $(0,u),(0,v) \in (\Z/N\Z)^{(n,2)}$.
See the subsection \S\ref{ss:Jacobi_group}
for the symbol of  the matrix
$
  \left(M\times
     \left(\begin{smallmatrix}N^2&0&0&0\\0&N&0&0\\0&0&1&0\\0&0&0&N\end{smallmatrix}\right),
     [((0,u),(0,v)),0_2]\right).
$
The above summations are finite sums and do not depend on
the choice of the representatives $u$, $v$ and $M$.
One can check that
$\psi|V(X)$ belongs to $J_{k,\mathcal{M}[\left(\begin{smallmatrix}N&0\\0&1\end{smallmatrix}\right)]}^{(n)}$.
It means that $V(X)$ is a map:
\begin{eqnarray*}
 V(X) \ : \ J_{k,\mathcal{M}}^{(n)} \rightarrow J_{k,\mathcal{M}[\left(\begin{smallmatrix}N&0\\0&1\end{smallmatrix}\right)]}^{(n)}.
\end{eqnarray*}
Moreover, if $\psi \in J_{k,\mathcal{M}}^{(n)\, cusp}$, then $\psi|V(X) \in J_{k,\mathcal{M}[\left(\begin{smallmatrix}N&0\\0&1\end{smallmatrix}\right)]}^{(n)\, cusp}$.
This fact is shown by the expression of the Fourier coefficients of  $\psi|V(X)$.

For the sake of simplicity we set
\begin{eqnarray*}
 V_{\alpha,n-\alpha}(p^2) 
 &:=&
 V(\mbox{diag}(1_{\alpha},p 1_{n-\alpha}, p^2 1_{\alpha}, p 1_{n-\alpha}))
\end{eqnarray*}
for any prime $p$ and for any $\alpha$ $(0\leq \alpha \leq n)$.

Next we shall define index-shift maps for Jacobi forms of half-integral weight of integer index.
We assume that $p$ is an odd prime.
Let $m$ be a positive integer.
Let $Y = (X,\varphi) \in \widetilde{\mbox{GSp}_n^+(\Z)}$ with $n(X) = p^{2\nu}$, where $\nu$ is a positive integer.
For $\phi \in J_{k-\frac12,m}^{(n)}$ we define
\begin{eqnarray*}
 \phi|\widetilde{V}(Y)
 &:=&
 n(X)^{\frac{n(2k-1)}{4} - \frac{n(n+1)}{2}}
 \!\!\!\!\!\!\!\!\!\!
 \sum_{\tilde{M} \in \Gamma_0^{(n)}(4)^* \backslash \Gamma_0^{(n)}(4)^* Y \Gamma_0^{(n)}(4)^*}
 \!\!\!\!\!\!\!\!\!\!
  \phi|_{k-\frac12,m} (\tilde{M},[(0,0),0]) ,
\end{eqnarray*}
where the above summation is a finite sum and
does not depend on the choice of the representatives $\tilde{M}$.
One can show by a direct computation
that $\phi|\widetilde{V}(Y)$ belongs to $J_{k-\frac12,mp^{2\nu}}^{(n)}$.

For the sake of simplicity we set
\begin{eqnarray*}
 \tilde{V}_{\alpha,n-\alpha}(p^2)
 &:=&
 \tilde{V}((\mbox{diag}(1_{\alpha},p 1_{n-\alpha}, p^2 1_{\alpha}, p 1_{n-\alpha}),p^{\alpha/2}))
\end{eqnarray*}
for any odd prime $p$ and for any $\alpha$ $(0\leq \alpha \leq n)$.

For the prime $p=2$,
the  index-shift map $\tilde{V}_{\alpha,n-\alpha}(4)$ is defined for certain subspace
$J_{k-\frac12,m}^{+(n)}$ of $J_{k-\frac12,m}^{(n)}$.
This is a map from $J_{k-\frac12,m}^{+(n)}$ to
$J_{k-\frac12,4m}^{+(n)}$.
The map $\tilde{V}_{\alpha,n-\alpha}(4)$ is defined through the linear isomorphism
between $J_{k,\M}^{(n)}$ and $J_{k-\frac12,m}^{+(n)}$
with $\M = \smat{ * }{ * }{ * }{1} \in L_2^+$ such that $\det(2\M) = m$.
The definition of $\tilde{V}_{\alpha,n-\alpha}(4)$ is
\begin{eqnarray*}
 \iota_{\mathcal{M}}(\psi)| \tilde{V}_{\alpha,n-\alpha}(4)
 &:=&
 2^{k(2n+1) - n (n +\frac72) + \frac12 \alpha }\
 \iota_{\mathcal{M}[\smat{2}{ }{ }{1}]}(\psi| V_{\alpha,n-\alpha}(4))
\end{eqnarray*}
for $\psi \in J_{k,\M}^{(n)}$.
Here $\iota_{\M}$ is the linear isomorphism map from 
$J_{k,\M}^{(n)}$ to $J_{k-\frac12,m}^{+(n)}$.
See \S\ref{s:lin_iso_jacobi} for the detail of the map $\iota_{\M}$.

\section{Ikeda lift and Siegel modular forms of half-integral weight}\label{s:Ikeda_lift}
In this section we recall the Ikeda lift and construct Siegel modular forms of half-integral weight
through the Eichler-Zagier-Ibukiyama correspondence.

Let $2n$ and $k$ be positive even integers such that $k > 2n+1$.
Let 
\begin{eqnarray*}
  g(z) &=&
  \sum_{\begin{smallmatrix} N \in \Z \\ N \equiv 0, (-1)^{n} \!\! \mod 4 \end{smallmatrix}}
  c(N) e(N z)  
    \in S_{k-n+\frac12}^{+(1)}
\end{eqnarray*}
be a Hecke eigenform.
Let $f$ be a normalized Hecke eigenform of elliptic cusp form of weight $2k-2n$
which corresponds to $g$ by the Shimura correspondence.
Let $F$ be a Ikeda lift of $g$ which is a Siegel cusp form of weight $k$ of degree $2n$
given by
\begin{eqnarray*}
  F(\tau) &=&
  \sum_{T \in L_{2n}^+} A(T) e(T \tau),
\end{eqnarray*}
where $\tau \in \H_{2n}$ and the Fourier coefficient $A(T)$ is given by
\begin{eqnarray*}
  A(T) &=&
  c(|D_T|) f_T^{k-n-\frac12} \prod_{\begin{smallmatrix} p| f_T \end{smallmatrix}} 
  \tilde{F}_{p}(T, \alpha_p),
\end{eqnarray*}
and where the fundamental discriminant $D_T$ and the natural number $f_T$ 
are determined by $(-1)^n \det(2T) = D_T f_T^2$,
and where $\tilde{F}_p(T,X) \in \C[X + X^{-1}]$  is a certain Laurent polynomial
(see \cite[p. 642]{Ik} for the definition).
Here $\alpha_p$ is the complex number determined by
$a_f(p) = (\alpha_p + \alpha_p^{-1})p^{k-n-\frac12}$, and where $a_f(p)$ is the $p$-th Fourier coefficient of $f$.

We take the Fourier-Jacobi expansion of $F$ : 
\begin{eqnarray*}
  F\left(\begin{pmatrix} \tau & z \\ ^t z & \omega \end{pmatrix}\right)
  &=&
  \sum_{N \in \Z} \psi_N(\tau,z) e(N \omega),
\end{eqnarray*}
where $\tau \in \H_{2n-1}$, $\omega \in \H_1$ and $z \in \C^{(2n-1,1)}$.
The function $\psi_N$ is a Jacobi form of weight $k$ of index $N$ of degree $2n-1$.

There exists $G \in S_{k-\frac12}^{+(2n-1)}$ which corresponds to $\psi_1$
by the linear isomorphism between $S_{k-\frac12}^{+(2n-1)}$ and the space
of Jacobi cusp forms of weight $k$ of index $1$ of degree $2n-1$
(see ~\cite{Ib} and see also Proposition~\ref{prop:matrix_integer_isom} in \S\ref{s:lin_iso_jacobi}.
The form $G$ is given by $G = \iota_1(\psi_1)$).

We remark that the map
\begin{eqnarray*}
   S_{k-n + \frac12}^{+(1)} \rightarrow S_{k-\frac12}^{+(2n-1)}
\end{eqnarray*}
given by the above manner is a linear map.
If $g$ is a Hecke eigenform, then $G$ is also a Hecke eigenform (cf. \cite[Theorem 1.1]{FJlift}).

\section{Generalization of the Maass relation for Siegel modular forms of half-integral weight}
\label{s:maass_relation}

Let $g \in S_{k-n+\frac12}^{+(1)}$ be a Hecke eigen form.
Let $G \in S_{k-\frac12}^{+(2n-1)}$ be the lift of $g$ as in \S\ref{s:Ikeda_lift}.
We take the Fourier-Jacobi expansion
\begin{eqnarray*}
  G\left(\begin{pmatrix} \tau & z \\ ^t z & \omega \end{pmatrix} \right)
  &=&
  \sum_{m\in \Z_{>0}} \phi_{m}(\tau,z) e(m\omega).
\end{eqnarray*}
Since $G$ belongs to the plus-space, $\phi_m = 0$ unless
$m \equiv 0$, $3$ $\!\! \mod 4$.

To explain a generalization of the Maass relation,
we prepare some symbols.
For integers $l$ $(2\leq l)$, $\beta$ $(0 \leq \beta \leq l-1)$
and $\alpha$ $(0\leq \alpha \leq l)$, we put
\begin{eqnarray*}
b_{\beta,\alpha} &:=&
b_{\beta,\alpha,l,p}(X)
=
\begin{cases}
  (p^{l +1 -\alpha} - p^{-l-1+\alpha})  p^{\frac12} & \mbox{if } \beta = \alpha-2, \\
  (X + X^{-1}) p & \mbox{if } \beta = \alpha-1, \\
  p^{-l+\alpha + \frac32}  & \mbox{if } \beta = \alpha, \\
  0 & \mbox{otherwise},
\end{cases}
\end{eqnarray*}
and we set $B_{l,l+1}(X)$ as the $l \times (l+1)$-matrix defined by
\begin{eqnarray*}
  B_{l,l+1}(X) &:=&
  (b_{\beta,\alpha})_{\begin{smallmatrix} \beta = 0,...,l-1 \\ \alpha = 0,...,l \end{smallmatrix}}
  =
  \begin{pmatrix}
    p^{-l+\frac32} & b_{0,1} & \cdots & b_{0,l} \\
     & \ddots & \ddots & \vdots \\
    0  &  & p^{\frac12} & b_{l-1,l}
  \end{pmatrix}
\end{eqnarray*}
with entries in $\C[X + X^{-1}]$. 
The $2\times (n+1)$-matrix $A_{2,2n-1}^{p}\!\left(X\right)$
is defined by
\begin{eqnarray*}
  A_{2,n+1}^{p}(X)
  &:=&
  \prod_{l=2}^{n} B_{l,l+1}(p^{\frac{n+2}{2}-l} X) \\
  &=&
  B_{2,3}(p^{\frac{n+2}{2}-2} X) B_{3,4}(p^{\frac{n+2}{2}-3}X) \cdots B_{n,n+1}(p^{\frac{n+2}{2}-n}X).
\end{eqnarray*}

For $a \in \Z_{>0}$ and
for $\phi \in J_{k-\frac12,m}^{(n')}$ we define the function $\phi|U_{a}$ by
\begin{eqnarray*}
  (\phi|U_{a})(\tau,z)
  &:=&
  \phi(\tau, a z ).
\end{eqnarray*}
In~\cite{CE_N} we obtained the following generalization of the Maass relation.
\begin{prop}\label{prop:maass_relation}
For any natural number $m$ and for any prime $p$,
we have the identity between the vectors
  \begin{eqnarray*}
  &&  \!\!\!\!\!\!\!\!\!\!
  (\phi_m|\tilde{V}_{0,2n-2}(p^2), \phi_m|\tilde{V}_{1,2n-3}(p^2),...,\phi_m|\tilde{V}_{2n-2,0}(p^2))\\
  &=&
  p^{k(2n-3)-2 n^2 - n + \frac{11}{2}}
  \left(\phi_{\frac{m}{p^2}}|U_{p^2},\, \phi_m|U_p,\, \phi_{mp^2} \right)
  \begin{pmatrix}
          0 & p^{2k-3} \\
          p^{k-2} & p^{k-2} \left(\frac{-m}{p}\right) \\
          0 & 1 \end{pmatrix}\\
  &&
  \times
  A_{2,2n-1}^{p}\!\left(\alpha_p\right)
  \,
  diag(1,p^{\frac12},p,...,p^{n-1}).
\end{eqnarray*}
\end{prop}

Therefore we obtain
\begin{eqnarray}\label{id:prop_4_1}
   && \\ \notag
   \phi_m| \tilde{V}_{i,2n-2-i}(p^2)
   &=&
   a_1 \, \phi_m|U_p
   +
   a_2 \left( p^{2k-3} \phi_{\frac{m}{p^2}} | U_{p^2} + p^{k-2} \left( \frac{-m}{p} \right) \phi_m|U_p + \phi_{mp^2} \right)
\end{eqnarray}
for any $i$ $(0 \leq i \leq 2n-2)$
with certain constants $a_j = a_{j,i,p}$ $(j = 1,2)$ which do not depend
on the choice of $m$.
In particular, if $i = 1$, then by a straightforward calculation
we obtain $a_2 = a_{2,1,p} = p^{k(2n-3)-4n^2+7n-\frac32} \neq 0$.

\begin{df}
 For any prime $p$
 we set
 \begin{eqnarray*}
   \hat{D}_{2n-2}(p^2)
   &:=&
   a_2^{-1} (\tilde{V}_{1,2n-3}(p^2) - a_1 U_p),
 \end{eqnarray*}
 where $a_1$ and $a_2$ are constants determined by the identity~(\ref{id:prop_4_1}).
\end{df}
We remark that the map $\hat{D}_{2n-2}(p^2)$ depend on the choice of $g$,
since $a_1$ is determined by the value $\{ \alpha_p^{\pm}\}$.
Remark also that $\hat{D}_{2n-2}(p^2)$ and $\hat{D}_{2n-2}(q^2)$ are compatible
if two primes $p$ and $q$ are not the same.
Moreover, $\hat{D}_{2n-2}(p^2)$ and $U_N$ are compatible for any natural number $N$.

Let $n' = 2n-2$
and let $m$ be a natural number.
We define the sequence of maps $\{\hat{D}_{n'}(N)\}_N$
and $\{D_{n'}(N)\}_N := \{D_{n',m}(N)\}_N$ through the following two formal Dirichlet series
\begin{eqnarray*}
 \sum_{N=1}^{\infty}  \frac{\hat{D}_{n'}(N^2)}{N^s}
 &=&
 \prod_{p}  \left( I - \hat{D}_{n'}(p^2) p^{-s}+ U_{p^2} p^{2k-3-2s}\right)^{-1} \\
\end{eqnarray*}
and
\begin{eqnarray*}
 \sum_{N=1}^{\infty}  \frac{D_{n'}(N^2)}{N^s}
 &=&
 \left\{\prod_{p}  \left(I - \left(\frac{-m}{p}\right) U_p p^{k-2-s} \right) \right\}
 \sum_{N=1}^{\infty}  \frac{\hat{D}_{n'}(N^2)}{N^s},
\end{eqnarray*}
where $I$ denotes the identity map.
We remark that
the definitions of $\hat{D}_{n'}(N^2)$
and $D_{n'}(N^2)$ depend on the choice of 
$g \in S_{k-n+\frac12}^{+(1)}$.
The definition of $\hat{D}_{n'}(N^2)$
is independent of the choice of $m$,
but the definition of $D_{n'}(N^2)$ depends on the choice of $m$.
Because of the compatibility of $\hat{D}_{n'}(p^2)$ and $U_N,$ the above definitions
are well-defined.
Moreover $D_{n'}(N^2)$ and $U_M$ are compatible for any natural numbers $N$ and $M$.

\begin{prop}\label{prop:maass_relation_D}
Let $G \in S_{k-\frac12}^{+(2n-1)}$ be as above.
For any natural numbers $N$ and
for any negative fundamental discriminant $-m$,
we define $D_{n'}(N^2)$ as above.
Then we have
\begin{eqnarray*}
  \phi_{mN^2} &=& \phi_m|D_{n'}(N^2).
\end{eqnarray*}
\end{prop}
\begin{proof}
Due to the definition of $D_{n'}(p^{2\delta})$, we have
\begin{eqnarray*}
  0 
  &=&
  p^{2k-3} D_{n'}(p^{2\delta-2}) U_{p^2}
  -  D_{n'}(p^{2\delta}) \hat{D}_{n'}(p^2)
  +   D_{n'}(p^{2\delta+2})
\end{eqnarray*}
for any $\delta \geq 1$
and
\begin{eqnarray*}
    - \left( \frac{-m}{p} \right) p^{k-2} U_{p}
  &=&
  -  \hat{D}_{n'}(p^2)
  +  D_{n'}(p^2).
\end{eqnarray*}
Thus
\begin{eqnarray}\label{id:prop_4_2} 
  && \\
  \notag
  0 
  &=&
  p^{2k-3} \phi_{m} | D_{n'}(p^{2\delta-2}) U_{p^2}
  - \phi_{m}| D_{n'}(p^{2\delta}) \hat{D}_{n'}(p^2)
  + \phi_{m} | D_{n'}(p^{2\delta+2})
\end{eqnarray}
for any $\delta \geq 1$ and
\begin{eqnarray}\label{id:prop_4_3}
  - \left( \frac{-m}{p} \right) p^{k-2} \phi_{m}|U_{p}
  &=&
  - \phi_{m}| \hat{D}_{n'}(p^2)
  + \phi_{m} | D_{n'}(p^2).
\end{eqnarray}

On the other hand, 
since $\hat{D}_{n'}(p^2) = - \frac{a_1}{a_2} U_p  + \frac{1}{a_2} \tilde{V}_{1,2n-3}(p^2)$
and due to (\ref{id:prop_4_1}) with $i = 1$, we have
\begin{eqnarray}\label{id:prop_3_4}
  - \left( \frac{-m'}{p} \right) p^{k-2} \phi_{m'}|U_{p}
  &=&
  p^{2k-3} \phi_{\frac{m'}{p^2}}|U_{p^2}
  - \phi_{m'}|\hat{D}_{n'}(p^2) + \phi_{m'p^2}
\end{eqnarray}
for any natural number $m'$.

We remark that $\phi_{\frac{m}{p^2}} = 0$.

If $m' = m$, then due to the identities (\ref{id:prop_4_3}) and (\ref{id:prop_3_4}),
we have $\phi_{mp^2} = \phi_m | D_{n'}(p^2)$.

If $m' = mp^{2\delta}$ $(\delta \geq 1)$ and if $\phi_{mp^{2\delta - 2j}} = \phi_m|D_{n'}(p^{2\delta - 2j}) $ $(j = 0,1)$
is true, then due to the identities (\ref{id:prop_4_2}) and (\ref{id:prop_3_4}),
we have $\phi_{mp^{2\delta+2}} = \phi_m|D_{n'}(p^{2\delta+2})$.
Thus, by induction we obtain $\phi_{mp^{2\delta}} = \phi_m|D_{n'}(p^{2\delta})$
for any $\delta \geq 0$.

Let $q$ be a prime which is different from $p$, then $\left( \frac{-mp^{2\delta}}{q} \right)
= \left( \frac{-m}{q} \right)$.
Thus the definition of $D_{n'}(q^{2\gamma})$ does not change, even if we
replace $m$ by $m p^{2\delta}$. The identities (\ref{id:prop_4_2}) and (\ref{id:prop_4_3})
are true, even if we replace $m$ and $p$ by $mp^{2\delta}$ and by $q$, respectively.
Therefore, we conclude this proposition by induction with respect to
natural numbers $p^{2\delta}$ and $m$.
\end{proof}

\begin{lemma}\label{lem:DD}
Let $G \in S_{k-\frac12}^{+(2n-1)}$ be as above.
We fix a natural number $m$. 
We define $\hat{D}_{n'}(N)$ and $D_{n'}(N)$ as above.
For any natural numbers $N$, $M$, $\delta$ and $\gamma$,
and for any prime $p$,
we have the identities
\begin{itemize}
\item[(1)]
 \begin{eqnarray*}
   \hat{D}_{n'}(p^{2\delta}) \hat{D}_{n'}(p^{2\gamma})
   &=&
   \sum_{i=0}^{min(\delta,\gamma)} p^{(2k-3)i} U_{p^{2i}} \hat{D}_{n'}(p^{2(\delta +  \gamma - 2i)})
 \end{eqnarray*}
\item[(2)]
 \begin{eqnarray*}
   \hat{D}_{n'}(N^2) \hat{D}_{n'}(M^2)
   &=&
   \sum_{d | (N,M)} d^{(2k-3)} U_{d^2} \hat{D}_{n'}\left(\frac{N^2 M^2}{d^4}\right)
 \end{eqnarray*}
\item[(3)]
\begin{eqnarray*}
  D_{n'}(p^{2\delta}) D_{n'}(p^{2 \gamma})
  &=&
  \sum_{i=0}^{min(\delta,\gamma)}
    p^{(2k-3)i} U_{p^{2i}} D_{n'}(p^{2(\delta + \gamma - 2i)})\\
  &&
  -
  \left( \frac{-m}{p}\right) p^{k-2}
  \sum_{i=0}^{min(\delta,\gamma)-1}
    p^{(2k-3)i} U_{p^{2i+1}} D_{n'}(p^{2(\delta + \gamma - 1 - 2i)}).
\end{eqnarray*}
\item[(4)]
\begin{eqnarray*}
  D_{n'}(N^2) D_{n'}(M^2)
  &=&
  \sum_{d | (N,M)}
    d^{2k-3} 
  \sum_{d_1 | \frac{(N,M)}{d}}
   \mu(d_1) d_1^{k-2} \left( \frac{-m}{d_1} \right) U_{d^2 d_1} D_{n'}\! \left( \frac{N^2 M^2}{d^4 d_1^2}\right),
\end{eqnarray*}
where $\mu$ is the M\"obius function.
\end{itemize}
\end{lemma}
\begin{proof}
 The identity (1) follows from a straightforward calculation. 
 The identity (3) follows from (1) and from the relation
 \begin{eqnarray*}
   D_{n'} (p^{2\delta})
   &=&
   \hat{D}_{n'}(p^{2\delta}) - \left( \frac{-m}{p} \right) p^{k-2} U_p \hat{D}_{n'}(p^{2\delta-2}).
 \end{eqnarray*}
 The identities (2) and (4) follow from (1) and (3).
\end{proof}

\section{Isomorphism between the spaces of Jacobi forms}\label{s:lin_iso_jacobi}
In this section we review a generalization of
the Eichler-Zagier-Ibukiyama correspondence shown in~\cite{matrix_integer}.
It is a linear isomorphism between certain spaces of Jacobi forms
of integral weight and of half-integral weight.

Let $n'$, $k$ and $r$ be natural numbers. 
We assume that $k$ is an even integer
and $r \geq 1$.
We take a matrix $\M = \begin{pmatrix} \M_1 & \frac12 L \\ \frac12 ^t L & 1  \end{pmatrix} \in L_r^+$
with $\M_1 \in L_{r-1}^+$ and $L \in M_{r-1,1}(\Z)$.
We set
\begin{eqnarray*}
  \mathfrak{M} &:=& 
  \begin{cases} 
    4\M_1 - L {^t L} & \mbox{if $r \geq 2$}, \\
    \emptyset & \mbox{if $r = 1$}.
  \end{cases}
\end{eqnarray*}
If $r = 1$, then $\M = 1$ and
we put $\det(\mathfrak{M}) := \det(\emptyset) = 1$ by abuse of notation.

A plus-space $J_{k-\frac12,\mathfrak{M}}^{+(n')}$ for Jacobi forms is introduced
in~\cite{matrix_integer}.
This is a subspace of $J_{k-\frac12,\mathfrak{M}}^{(n')}$
and is a generalization of the generalized plus-space of Siegel modular forms
to Jacobi forms,
and where $J_{k-\frac12,\mathfrak{M}}^{(n')}$ denotes the space of Jacobi forms
of weight $k-\frac12$ of index $\mathfrak{M}$ of degree $n'$.
(See the definition of $J_{k-\frac12,\mathfrak{M}}^{(n')}$ in~\cite[\S2.3]{matrix_integer} for the case $r-1 \geq 2$.
In this article we use the case $r-1 = 0$ and $r-1 = 1$).
The space $J_{k-\frac12,\mathfrak{M}}^{+(n')}$ is defined as follows.
Let $\phi \in J_{k-\frac12, \mathfrak{M}}^{(n')}$
be a Jacobi form of weight $k-\frac12$ of index $\mathfrak{M}$ on $\Gamma_0^{(n')}(4)$.
We take the Fourier expansion of $\phi$ :
\begin{eqnarray*}
 \phi(\tau,z) &=&  \sum_{N',R'} C_\phi(N',R') e(N' \tau + R' {^t z})
\end{eqnarray*}
for $(\tau,z) \in \H_{n'} \times \C^{(n',r-1)}$,
where $N'$ and $R'$ run over $L_{n'}^*$ and $\Z^{(n',r-1)}$, respectively, such that
$4 N' - R' \mathfrak{M}^{-1} {^t R'} \geq 0$.
Then $\phi$ belongs to $J_{k-\frac12,\mathfrak{M}}^{+(n')}$ if and only if 
$C_\phi(N',R') = 0$ unless
\begin{eqnarray*}
  \begin{pmatrix}
    N' & \frac12 R' \\ \frac12 {^t R'} & \mathfrak{M}
  \end{pmatrix}
  + \lambda {^t \lambda}  \in 4 L_{n'}^*
\end{eqnarray*}
with some $\lambda \in \Z^{(n'+r-1,1)}$.

This condition requires a condition on $\mathfrak{M}$.
For example, if $r-1 = 1$ and if  $J_{k-\frac12,\mathfrak{M}}^{+(n')} \neq \emptyset$,
then $\mathfrak{M} \equiv 0, 3 \!\! \mod 4$.

We define
$J_{k-\frac12,\mathfrak{M}}^{+(n')\, cusp} 
:= J_{k-\frac12,\mathfrak{M}}^{+(n')} \cap J_{k-\frac12,\mathfrak{M}}^{(n')\, cusp}$.

If $r = 1$, then
$J_{k-\frac12,\mathfrak{M}}^{+(n')} = J_{k-\frac12,\emptyset}^{+(n')} = M_{k-\frac12}^{+(n')}$
and
$J_{k-\frac12,\mathfrak{M}}^{+(n')\, cusp} = J_{k-\frac12,\emptyset}^{+(n')\, cusp} 
= S_{k-\frac12}^{+(n')}$.

\begin{lemma}
 Let $F \in M_{k-\frac12}^{+(n'+r)}$. We take the Fourier-Jacobi expansion
 \begin{eqnarray*}
   F\left(\begin{pmatrix} \tau & z \\ ^t z & \omega \end{pmatrix} \right)
   &=&
   \sum_{\mathfrak{M} \in L_r^*} f_{\mathfrak{M}}(\tau,z) e(\mathfrak{M} \omega),
 \end{eqnarray*}
 where $\tau \in \H_{n'}$, $z \in \C^{(n',r)}$ and $\omega \in \H_r$.
 Then $f_{\mathfrak{M}} \in J_{k-\frac12,\mathfrak{M}}^{+(n')}$.
 Moreover, if $F \in S_{k-\frac12}^{+(n'+r)}$, 
 then $f_{\mathfrak{M}} \in J_{k-\frac12,\mathfrak{M}}^{+(n')\, cusp}$.
\end{lemma}
\begin{proof}
It is obvious from the definition of the plus-space $J_{k-\frac12,\mathfrak{M}}^{+(n')}$.
\end{proof}

In particular the function $\phi_m$ in \S\ref{s:maass_relation} belongs to the
plus-space $J_{k-\frac12,m}^{+(2n-2)}$. 

Let $\M$ be as above.
There exists a  linear isomorphism map $\iota_\M$ from $J_{k,\M}^{(n)}$
to $J_{k-\frac12,\mathfrak{M}}^{(n)+}$ (cf.~\cite{EZ} (for $r=n=1$), ~\cite{Ib} (for $r =1$, $n > 1$), ~\cite{matrix_integer} (for $r > 1$, $n \geq 1$)).
This map $\iota_\M$ is given as follows.

We assume $\psi \in J_{k,\M}^{(n')}$
and denote by $C_{\psi}( * , * )$ the Fourier coefficients of $\psi$.
For $\tau \in \H_{n'}$ and for $z = (z_1,z_2) \in \C^{(n',r)}$ ($z_1 \in \C^{(n',r-1)}$, $z_2 \in \C^{(n',1)}$),
we take the theta decomposition
\begin{eqnarray*}
  \psi(\tau,z)
  &=&
   \sum_{\begin{smallmatrix} R \in \Z^{(n',1)} \\ R \!\! \mod 2 \Z^{(n',1)} \end{smallmatrix}}
   f_{R}(\tau,z_1) \vartheta_{1,R,L}(\tau,z_1,z_2),
\end{eqnarray*}
where
\begin{eqnarray*}
  f_{R}(\tau,z_1)
  &=&
  \sum_{N_1 \in L_{n'}^*, N_3 \in \Z^{(n',r-1)}}
  C_{\psi}(N_1, \begin{pmatrix}  N_3 &  R  \end{pmatrix})
  \\
  && \times
  e( (N_1 - \frac14 R {^t R})\tau + (N_3 - \frac12 R {^t L}) {^t z_1})
\end{eqnarray*}
and where the function $\vartheta_{1,R,L}$ is defined by
\begin{eqnarray*}
   \vartheta_{1,R,L}(\tau,z_1,z_2)
   &:=&
   \vartheta_{1,\left(\begin{smallmatrix} R \\ L \end{smallmatrix} \right)}(\tau, \frac12 z_1 L + z_2) \\
   &=&
     \sum_{\begin{smallmatrix}
                 x \in \Z^{(n',1)} \\
                x \equiv \left(\begin{smallmatrix} R \\ L \end{smallmatrix} \right) \!\! \mod 2
                \end{smallmatrix}}
  e\left(\frac14 x  {^t x} \tau + x {^t \left( \frac12 z_1 L + z_2 \right)} \right)
\end{eqnarray*}
 (cf. \cite[Lemma 4.1]{matrix_integer}).
We put
\begin{eqnarray*}
  \iota_\M(\psi)(\tau,z_1)
  & := & 
  \sum_{R \in \Z^{(n,1)}/(2 \Z^{(n,1)})} f_{R}(4\tau, 4 z_1).
\end{eqnarray*}
Then $\iota_\M(\psi)$ belongs to $J_{k-\frac12,\mathfrak{M}}^{+(n')}$
(cf.~\cite[Proposition 4.4]{matrix_integer}).
If $\psi$ is a Jacobi cusp form, then $\iota_\M(\psi)$ is also a Jacobi cusp form.
If $r=1$, then $\M = 1$ and $\iota_1(\psi)$ is a Siegel modular form
(cf. \cite{EZ}, \cite{Ib}).
\begin{prop}[\cite{matrix_integer}]\label{prop:matrix_integer_isom}
We take a matrix $\M = \begin{pmatrix} \M_1 & \frac12 L \\ \frac12 ^t L & 1  \end{pmatrix} \in L_r^+$.
Let $k$ be an even integer.
The map $\iota_{\M} $ gives the linear isomorphisms:
\begin{eqnarray*}
  J_{k,\M}^{(n')} &\cong& J_{k-\frac12,\mathfrak{M}}^{+(n')}, \\
  J_{k,\M}^{(n')\, cusp} &\cong& J_{k-\frac12,\mathfrak{M}}^{+(n')\, cusp}.
\end{eqnarray*}
In the case of $r =1$ (it means $\M = 1$, $\mathfrak{M} = \emptyset$ and $J_{k-\frac12,\mathfrak{M}}^{+(n')} = M_{k-\frac12}^{+(n')}$),
these isomorphisms have been shown
in~\cite{EZ} (for $n'=1$) and in~\cite{Ib} (for $n' > 1$).
\end{prop}

Aa for the relation between the Petersson inner products and the linear isomorphism
map $\iota_\M$ is known as follows.

\begin{lemma}\label{lem:iota_petersson}
For $\psi_i \in J_{k,\M}^{(n')\, cusp}$ $(i=1,2)$
we set $\phi_i = \iota_{\M}(\psi_i) \in J_{k-\frac12,\mathfrak{M}}^{+ (n')\, cusp}$.
Then we have
\begin{eqnarray*}
  \langle \psi_1, \psi_2 \rangle
  &=&
  2^{2n'(k-1)}
  \langle \phi_1, \phi_2 \rangle .
\end{eqnarray*}
\end{lemma}
\begin{proof}
This lemma has been shown in~\cite[Theorem 5.4]{EZ} (for $r = 1$ and $n'=1$),
in \cite[p.2051]{KaKa} (for $r = 1$ and $n' > 1$),
and in \cite[Lemma 3.1]{RSJacobi} (for $r > 1$ and $n' \geq 1$).
\end{proof}

In the followings we consider the case $r = 2$.

Let $a \in \Qq_{>0}$, $b \in \Qq$ and $\M \in L_2^+$.
For $\psi \in J_{k,\M}^{(n')}$ we define the function
\begin{eqnarray*}
  (\psi|U_{\left(\begin{smallmatrix} a & 0 \\ b & 1 \end{smallmatrix} \right)})(\tau,z)
  &:=&
  \psi(\tau,z \left( \begin{smallmatrix} a & 0 \\ b & 1 \end{smallmatrix} \right))
\end{eqnarray*}

\begin{lemma}\label{lem:iota_u}
Let $r = 2$ and $\M = \left( \begin{smallmatrix} * & *  \\ * & 1 \end{smallmatrix} \right) \in L_2^+$. 
Let $\psi \in J_{k,\M}^{(n')}$.
Then, for any $a\in \Z_{>0}$ and for any $b\in \Z$,
we have $\psi|U_{\left(\begin{smallmatrix} a & 0 \\ b & 1 \end{smallmatrix} \right)}
\in J_{k,\M \left[\left(\begin{smallmatrix} a & 0 \\ b & 1 \end{smallmatrix} \right) \right]}^{(n')}$
and we have
\begin{eqnarray*}
  \iota_{\M}(\psi) | U_a
  &=&
  \iota_{\M \left[\left(\begin{smallmatrix} a & 0 \\ b & 1 \end{smallmatrix} \right) \right]}
  \left( \psi|U_{\left(\begin{smallmatrix} a & 0 \\ b & 1 \end{smallmatrix} \right)} \right).
\end{eqnarray*}
\end{lemma}
\begin{proof}
The first statement follows directly from the definition of Jacobi forms.
And the second statement can be shown by comparing the Fourier coefficients
of both sides.
The reader is referred to~\cite[Proposition 4.3]{CE_N} for the detail of the calculation
for the second statement.
\end{proof}

\begin{lemma}\label{lem:iota_v}
Let $r = 2$ and $\M = \left( \begin{smallmatrix} * & *  \\ * & 1 \end{smallmatrix} \right) \in L_2^+$. 
For any odd prime $p$ and for $0\leq \alpha \leq n'$,
let $\tilde{V}_{\alpha,n'-\alpha}(p^2)$ and $V_{\alpha,n'-\alpha}(p^2)$ be index-shift maps
defined in \S\ref{ss:hecke_operators}.
Then, for any $\psi \in J_{k,\mathcal{M}}^{(n')}$ we have
\begin{eqnarray}\label{id:iota_hecke}
 \iota_{\mathcal{M}}(\psi)| \tilde{V}_{\alpha,n'-\alpha}(p^2)
 &=&
 p^{k(2n'+1) - n' (n'+\frac72) + \frac12 \alpha }\
 \iota_{\mathcal{M}[\smat{p}{ }{ }{1}]}(\psi| V_{\alpha,n'-\alpha}(p^2)).
\end{eqnarray}
\end{lemma}
\begin{proof}
The identity can be shown by comparing the Fourier coefficients of both sides.
The proof is the same as in~\cite[Proposition 4.4]{CE_N}.
\end{proof}
We remark that $\tilde{V}_{\alpha,n'-\alpha}(4)$ has been defined through
the identity~(\ref{id:iota_hecke}).
(See \S\ref{ss:hecke_operators}).

\section{Adjoint maps}\label{s:adjoint_maps}
In this section we introduce some adjoint maps for Jacobi forms
with respect to the Petersson inner product.

We assume $\mathcal{M} = \smat{*}{*}{ * }{1} \in L_2^+$ such that
$\det(2\M) = m$.
We remark that $m \equiv 0, 3 \!\! \mod 4$.
There exists $\M$ for any such natural number $m$.

Let $n'$ and $N$ be natural numbers.

\subsection{Adjoint map $U^*$ for Jacobi forms of integral weight}

For any $\psi \in J_{k,\mathcal{M}\left[\left(\begin{smallmatrix} N & 0 \\ 0 & 1 \end{smallmatrix} \right) \right]}^{(n')}$ we define the function
\begin{eqnarray*}
  \psi|U^*_{\left(\begin{smallmatrix} N & 0 \\ 0 & 1 \end{smallmatrix} \right) }
  &:=&
  N^{-2n'}
  \sum_{\lambda_1,\mu_1 \in (\Z/ N\Z)^{(n',1)}}
  \psi|U_{\left(\begin{smallmatrix} N^{-1} & 0 \\ 0 & 1 \end{smallmatrix} \right)} |_{k,\mathcal{M}}
  [((\lambda_1,0),(\mu_1,0)),0_2].
\end{eqnarray*}
\begin{lemma}
We obtain a map
\begin{eqnarray*}
  U^*_{\left(\begin{smallmatrix} N & 0 \\ 0 & 1 \end{smallmatrix} \right) }
  \ : \
  J_{k,\mathcal{M}\left[\left(\begin{smallmatrix} N & 0 \\ 0 & 1 \end{smallmatrix} \right) \right]}^{(n')}
  \rightarrow
  J_{k,\mathcal{M}}^{(n')}.
\end{eqnarray*}
Moreover,
if $\psi \in J_{k,\mathcal{M}\left[\left(\begin{smallmatrix} N & 0 \\ 0 & 1 
       \end{smallmatrix} \right) \right]}^{(n')\, cusp}$, 
       then $\psi|U^*_{\left(\begin{smallmatrix} N & 0 \\ 0 & 1 \end{smallmatrix} \right) }
\in J_{k,\mathcal{M}}^{(n')\, cusp}$.
\end{lemma}
\begin{proof}
  One can check the first statement by a straightforward calculation.
  For the second statement we need to show that $\psi|U^*_{\left(\begin{smallmatrix} N & 0 \\ 0 & 1 \end{smallmatrix} \right) }$
  is a Jacobi cusp form.
  It is shown by the expression of the Fourier coefficients of
  $\psi|U^*_{\left(\begin{smallmatrix} N & 0 \\ 0 & 1 \end{smallmatrix} \right) }$
  by using the Fourier coefficients of $\psi$.
\end{proof}

\begin{lemma}\label{lem:adjoint_U}
  The map $U^*_{\left(\begin{smallmatrix} N & 0 \\ 0 & 1 \end{smallmatrix} \right) }$
  is the adjoint map of
  $U_{\left(\begin{smallmatrix} N & 0 \\ 0 & 1 \end{smallmatrix} \right)}$
   with respect to the Petersson inner product.
  It means that 
 we have
 \begin{eqnarray*}
   \langle \psi_1|
   U_{\left(\begin{smallmatrix} N & 0 \\ 0 & 1 \end{smallmatrix} \right)}, \psi_2 \rangle
   &=&
   \langle \psi_1, 
   \psi_2|U^*_{\left(\begin{smallmatrix} N & 0 \\ 0 & 1 \end{smallmatrix} \right) } \rangle
 \end{eqnarray*}
 for any $\psi_1 \in 
 J_{k,\mathcal{M}}^{(n')\, cusp}$ and
 any $\psi_2 \in 
 J_{k,\mathcal{M}\left[\left(\begin{smallmatrix} N & 0 \\ 0 & 1 \end{smallmatrix}
  \right) \right]}^{(n')\, cusp}$.
\end{lemma}
\begin{proof}
We fix a fundamental domain $\mathcal{F}_{n',2}$ of $\Gamma_{n',2}^J \backslash (\H_{n'} \times \C^{(n',2)})$.
And we put
\[
\mathcal{F}_{n',2}(N) = \left\{ (\tau,z\left(\begin{smallmatrix} N & 0 \\ 0 & 1 \end{smallmatrix} \right)) \in \H_{n'} \times \C^{(n',2)} \, | \, (\tau,z) \in \mathcal{F}_{n',2} \right\}.
\]
We have $\mbox{vol}(\mathcal{F}_{n',2}(N)) = N^{2n'}\mbox{vol}(\mathcal{F}_{n',2})$.
Here $\mbox{vol}(*)$ is the volume with the measure $\det(v)^{-n'-3} du\, dv\, dx\, dy$.
Let $X = \left[((\lambda_1,0), (\mu_1,0)),0_2\right] \in \Z^{(n',2)} \times \Z^{(n',2)} \times \Z^{(2,2)}$.
We write $z = x + \sqrt{-1} y =  (z_1,z_2) \in \C^{(n',2)}$,
$z_i = x_i + \sqrt{-1} y_i$ $(i=1,2)$.
Then, by the substitution $ z_1 \rightarrow z_1 + \frac{1}{N} \tau \lambda_1 + \frac{1}{N} \mu_1$,
we have
\begin{eqnarray*}
    &&
    \langle \psi_1|U_{\left(\begin{smallmatrix} N & 0 \\ 0 & 1 \end{smallmatrix} \right)}, \psi_2 \rangle 
    \\
    &=&
    \int_{\mathcal{F}_{n',2}} \psi_1(\tau,(N z_1, z_2))
     \overline{\psi_2(\tau,z)} \det(v)^{k-n'-3}
     e(\frac{-4 \pi}{2\pi i}  \M\left[\left(\begin{smallmatrix} N & 0 \\ 0 & 1 \end{smallmatrix} \right) \right] v^{-1}[y]) \, du\, dv\, dx\, dy \\
    &=&
    \int_{\mathcal{F}_{n',2}} \psi_1(\tau,(Nz_1+\tau \lambda_1 +  \mu_1, z_2))
     \overline{\psi_2(\tau,(z_1+ \frac{1}{N}\tau \lambda_1 + \frac{1}{N}\mu_1,z_2))} \det(v)^{k-n'-3} \\
     &&
     \times e(\frac{-4 \pi}{2\pi i} 
     \M\left[\left(\begin{smallmatrix} N & 0 \\ 0 & 1 \end{smallmatrix} \right) \right] 
     v^{-1}[y+ (\frac{1}{N}v \lambda_1,0)]) \, du\, dv\, dx\, dy,
\end{eqnarray*}
 and by the substitution $z_1 \rightarrow \frac{1}{N}z_1$,
\begin{eqnarray*}
    &=&
    \int_{\mathcal{F}_{n',2}(N)}
    N^{-2n'} \psi_1(\tau,(z_1+\tau\lambda_1+ \mu_1,z_2)) \overline{\psi_2(\tau,(N^{-1}(z_1+\tau\lambda_1
     + \mu_1),z_2))} \det(v)^{k-n'-3} \\
    &&
    \times e(\frac{-4 \pi}{2\pi i}  \M v^{-1}[y+ (v \lambda_1,0)]) \, du\, dv\, dx\, dy \\
    &=&
    \int_{\mathcal{F}_{n',2}(N)}
    N^{-2n'} \psi_1(\tau,z) \overline{(\psi_2(*, * \left(\begin{smallmatrix} N^{-1} & 0 \\ 0 & 1 \end{smallmatrix}\right) )
    |X )(\tau,z)} \det(v)^{k-n'-3} e(\frac{-4 \pi}{2\pi i}  \M v^{-1}[y]) \, du\, dv\, dx\, dy.
\end{eqnarray*}
Thus we have
\begin{eqnarray*}
    &&
    \langle \psi_1|U_{\left(\begin{smallmatrix} N & 0 \\ 0 & 1 \end{smallmatrix} \right)}, \psi_2 \rangle  \\
    &=&
    N^{-2n'}\sum_{\lambda_1,\mu_1 \in \Z^{(n',1)}}
        \int_{\mathcal{F}_{n',2}(N)}
    N^{-2n'} \psi_1(\tau,z) \overline{(\psi_2(*, * \left(\begin{smallmatrix} N^{-1} & 0 \\ 0 & 1 \end{smallmatrix}\right) )
    |X )(\tau,z)} \\
    &&
    \times \det(v)^{k-n'-3} e(\frac{-4 \pi}{2\pi i}  \M v^{-1}[y]) \, du\, dv\, dx\, dy.
    \\
    &=&
    N^{-2n'} \int_{\mathcal{F}_{n',2}(N)}
     \psi_1(\tau,z) \overline{(\psi_2|U_{\left(\begin{smallmatrix} N & 0 \\ 0 & 1 \end{smallmatrix} \right)}^* )(\tau,z)} \det(v)^{k - n' - 3} e(\frac{-4 \pi}{2\pi i}  \M v^{-1}[y]) \, du\, dv\, dx\, dy  \\
    &=&
    \int_{\mathcal{F}_{n',2}}
     \psi_1(\tau,z) \overline{(\psi_2|U_{\left(\begin{smallmatrix} N & 0 \\ 0 & 1 \end{smallmatrix} \right)}^* )(\tau,z)} \det(v)^{k - n' - 3} e(\frac{-4 \pi}{2\pi i}  \M v^{-1}[y]) \, du\, dv\, dx\, dy \\
     &=&
         \langle \psi_1 , \psi_2 | U^*_{\left(\begin{smallmatrix} N & 0 \\ 0 & 1 \end{smallmatrix} \right)}\rangle .
\end{eqnarray*}
Thus we conclude this lemma.
\end{proof}

\begin{lemma}\label{lem:U*_UU*}
 For any natural numbers $N$ and $M$, we have
 \begin{itemize}
 \item[(1)]
 \begin{eqnarray*}
   \psi_1|U_{\smat{N}{0}{0}{1} }|U^*_{\smat{N}{0}{0}{1}}
   &=&
   \psi_1 ,
 \end{eqnarray*} 
 \item[(2)]
 \begin{eqnarray*}
   \psi_2|U^*_{\smat{N}{0}{0}{1} }
   &=&
   \psi_2|U_{\smat{M}{0}{0}{1} }|U^*_{\smat{NM}{0}{0}{1}} 
 \end{eqnarray*}
 \end{itemize}
 for any prime $p$ and for any $\psi_1 \in J_{k,\M}^{(n')}$ and any
 $\psi_2 \in J_{k,\M\left[\smat{N}{0}{0}{1} \right]}^{(n')}$.
 Moreover, if a natural number $d$ is coprime to $N$, then
 \begin{itemize}
  \item[(3)]
  \begin{eqnarray*}
   \psi_2|U_{\smat{d}{0}{0}{1}} | U^*_{\smat{N}{0}{0}{1} }
   &=&
   \psi_2}|U^*_{\smat{N}{0}{0}{1}} |U_{\smat{d}{0}{0}{1} , 
  \end{eqnarray*}
  \item[(4)]
  \begin{eqnarray*}
   \psi_3|U^*_{\smat{d}{0}{0}{1}} | U^*_{\smat{N}{0}{0}{1} }
   &=&
   \psi_3}|U^*_{\smat{N}{0}{0}{1}} |U^*_{\smat{d}{0}{0}{1} 
  \end{eqnarray*}  
  for any  $\psi_2 \in J_{k,\M\left[\smat{N}{0}{0}{1} \right]}^{(n')}$ and any
  $\psi_3 \in J_{k,\M\left[\smat{Nd}{0}{0}{1} \right]}^{(n')}$.
 \end{itemize}
\end{lemma}
\begin{proof}
 The identity (1) is obvious because of the definitions.
 The identity (2) follows from the definitions of $U_{\smat{N}{0}{0}{1} }$ and $U^*_{\smat{N}{0}{0}{1}}$,
 since
 \begin{eqnarray*}
   &&
   \psi|U_{\smat{M}{0}{0}{1} }|U_{\smat{(NM)^{-1}}{0}{0}{1}}
   |_{k,\M}[(\lambda_1 + N \lambda'_1,0), (\mu_1 + N \mu'_1,0),0_2] \\
   &=&
   \psi|U_{\smat{N^{-1}}{0}{0}{1}}
   |_{k,\M}[(\lambda_1,0), (\mu_1,0),0_2] 
    \end{eqnarray*}
    for any $\lambda_1$, $\lambda'_1$, $\mu_1$, $\mu'_1$ $\in \Z^{(n',1)}$.
    The identities (3) and (4) also follow from straightforward calculations.
\end{proof}

\subsection{Adjoint map $U^*$ for Jacobi forms of half-integral weight}\label{ss:adjoint_map_of_u_half}

\begin{df}\label{df:U*_N}
By combining Proposition~\ref{prop:matrix_integer_isom}
and Lemma~\ref{lem:adjoint_U} we define
the adjoint map $U_N^*$ of $U_N$ with respect to the Petersson inner product.
For $\phi \in J_{k-\frac12, mN^2}^{+(n')}$ we define
\begin{eqnarray*}
  \phi|U_N^* &:=& \iota_{\M}((\iota_{\M\left[ \smat{N}{0}{0}{1} \right]}^{-1} (\phi))|U_{\smat{N}{0}{0}{1} }^*).
\end{eqnarray*}
Then, due to Proposition~\ref{prop:matrix_integer_isom}
and Lemma~\ref{lem:adjoint_U},  we have $\phi|U_N^* \in J_{k-\frac12, m}^{+(n')}$.
And if $\phi \in J_{k-\frac12, mN^2}^{+(n')\, cusp}$, then $\phi|U_N^* \in J_{k-\frac12, m}^{+(n')\, cusp}$.
\end{df}

\begin{lemma}\label{lem:u_adj}
    The map $U^*_N$
  is the adjoint map of
  $U_N$
   with respect to the Petersson inner product.
  It means that 
 we have
 \begin{eqnarray*}
   \langle \phi_1|
   U_N, \phi_2 \rangle
   &=&
   \langle \phi_1, 
   \phi_2|U^*_N \rangle
 \end{eqnarray*}
 for any $\phi_1 \in 
 J_{k-\frac12,m}^{(n')\, cusp}$ and any
 $\phi_2 \in 
 J_{k-\frac12,mN^2}^{(n')\, cusp}$.
\end{lemma}
\begin{proof}
Due to Lemma~\ref{lem:iota_petersson}, Lemma~\ref{lem:iota_u} and Lemma~\ref{lem:adjoint_U},
we have
\begin{eqnarray*}
  \langle \phi_1|
   U_N, \phi_2 \rangle &=&
   2^{-2n'(k-1)}
   \langle \iota_{\M[\smat{N}{0}{0}{1}]}^{-1}(\phi_1|
   U_N), \iota_{\M[\smat{N}{0}{0}{1}]}^{-1}(\phi_2) \rangle \\
   &=&
   2^{-2n'(k-1)}
   \langle \iota_{\M}^{-1}(\phi_1)|
   U_{[\smat{N}{0}{0}{1}]}, \iota_{\M[\smat{N}{0}{0}{1}]}^{-1}(\phi_2) \rangle \\
      &=&
   2^{-2n'(k-1)}
   \langle \iota_{\M}^{-1}(\phi_1),
   \iota_{\M[\smat{N}{0}{0}{1}]}^{-1}(\phi_2)|U^*_{[\smat{N}{0}{0}{1}]} \rangle \\
   &=&
  \langle \phi_1, \phi_2|
   U^*_N \rangle .
\end{eqnarray*}
\end{proof}

\begin{lemma}\label{lem:u_adj_2}
  For any natural numbers $N$ and $M$,
 we have
 \begin{itemize}
 \item[(1)]
  \begin{eqnarray*}
    \phi_1|U_N | U^*_{N}
    &=&
    \phi_1 ,
  \end{eqnarray*} 
 \item[(2)]
  \begin{eqnarray*}
    \phi_2|U^*_N &=&
    \phi_2|U_M | U^*_{NM}
  \end{eqnarray*}
  \end{itemize}
  for any $\phi_1 \in J_{k-\frac12, m}^{+(n')}$ and any $\phi_2 \in J_{k-\frac12, mN^2}^{+(n')}$.
   Moreover, if a natural number $d$ is coprime to $N$, then
   \begin{itemize}
     \item[(3)]
  \begin{eqnarray*}
    \phi_2 | U_d | U^*_N &=&
    \phi_2 | U^*_N | U_d ,
  \end{eqnarray*}     
     \item[(4)]
  \begin{eqnarray*}
    \phi_3 | U^*_d | U^*_N &=&
    \phi_3 | U^*_N | U^*_d 
  \end{eqnarray*}
   \end{itemize}
   for any $\phi_2 \in J_{k-\frac12, mN^2}^{+(n')}$ and any $\phi_3 \in J_{k-\frac12, mN^2M^2}^{+(n')}$.
\end{lemma}
\begin{proof}
This lemma follows from Definition~\ref{df:U*_N},
Lemma~\ref{lem:iota_u} and Lemma~\ref{lem:U*_UU*}.
\end{proof}

\begin{lemma}
  If $(N,2)  =1 $, then the map $U_N^*$ is given by
\begin{eqnarray*}
  \phi|U_N^*
  &=&
  (2N)^{-2n'} \sum_{\lambda \in (\Z/ N \Z)^{n'}} \sum_{\mu \in (\Z/ 4 N \Z)^{n'}}
   \left(\phi|U_{\frac{1}{N}}\right)|_{m} \left[ (\lambda, \frac14 \mu) , 0 \right]
\end{eqnarray*}
  for any $\phi \in J_{k-\frac12, mN^2}^{+(n')}$.
\end{lemma}
\begin{proof}
 By comparing the Fourier coefficients of the both side, we obtain the identity.
\end{proof}

\subsection{Adjoint map $V^*$ for Jacobi forms of integral weight}

Let $X \in \mbox{GSp}_n^+(\Z)$ be a matrix such that the similitude of $X$ is $n(X)=N^2$
with a natural number $N$.
For any $\psi \in J_{k,\mathcal{M}\left[\left(\begin{smallmatrix} N & 0 \\ 0 & 1 \end{smallmatrix} \right) \right]}^{(n')}$ we define the function
\begin{eqnarray*}
 \psi|V^*(X)
 &:=&
 N^{2k-4n'}
 \sum_{\lambda_1,\mu_1 \in (\Z/N^2\Z)^{(n',1)}}
 \sum_{\lambda_2,\mu_2 \in (\Z/N\Z)^{(n',1)}}
 \sum_{M \in \Gamma_{n'} \backslash \Gamma_{n'} X \Gamma_{n'}} \!\!\!\!\! \\
 && \times
  \psi|_{k,\mathcal{M}} 
   \left(M\times
     \left(\begin{smallmatrix}1&0&0&0\\0&N&0&0\\0&0&N^2&0\\0&0&0&N\end{smallmatrix}\right),
     [((\lambda_1,\lambda_2),(\mu_1,\mu_2)),0_2]\right),
\end{eqnarray*}
where $(\lambda_1,\lambda_2),(\mu_1,\mu_2) \in (\Z/N^2\Z)^{(n',1)} \times (\Z/N\Z)^{(n',1)}$.
See \S\ref{ss:Jacobi_group}
for the symbol of  the matrix
$
  \left(M\times
     \left(\begin{smallmatrix}1&0&0&0\\0&N&0&0\\0&0&N^2&0\\0&0&0&N\end{smallmatrix}\right),
     [(\lambda,\mu),0_2]\right).
$
The above summations are finite sums and do not depend on
the choice of the representatives $(\lambda_1,\lambda_2)$, $(\mu_1,\mu_2)$ and $M$.

Let $V(X)$ be the index-shift map for Jacobi forms of matrix index
defined in~\S\ref{ss:hecke_operators}.
\begin{lemma}\label{lem:V*_VU*}
We have
\begin{eqnarray*}
   \psi|V^*(X) &=& \psi|V(X)|U^*_{\smat{N^2}{0}{0}{1} }
\end{eqnarray*}
In particular, the function $\psi|V^*(X)$ belongs to $J_{k,\mathcal{M}}^{(n')}$.
It means that $V^*(X)$ is a map:
\begin{eqnarray*}
 V^*(X) \ : \  J_{k,\mathcal{M}[\left(\begin{smallmatrix}N&0\\0&1\end{smallmatrix}\right)]}^{(n')}
  \rightarrow J_{k,\mathcal{M}}^{(n')}.
\end{eqnarray*}
Moreover, if $\psi \in J_{k,\mathcal{M}[\left(\begin{smallmatrix}N&0\\0&1\end{smallmatrix}\right)]}^{(n')\, cusp}$,
then $\psi|V^*(X) \in J_{k,\mathcal{M}}^{(n')\, cusp}$
\end{lemma}
\begin{proof}
The fist identity follows from the definitions of $V^*(X)$, $V(X)$ and $U^*_{\smat{N^2}{0}{0}{1}}$.
The other statements follow from this identity.
\end{proof}

\begin{lemma}\label{lem:adjoint_V}
  The map $V^*(X)$ is the adjoint map of
  $V(X)$ with respect to the Petersson inner product.
  It means that 
 we have
 \begin{eqnarray*}
   \langle \psi_1|V(X) , \psi_2 \rangle
   &=&
   \langle \psi_1, 
   \psi_2|V^*(X) \rangle
 \end{eqnarray*}
 for any $\psi_1 \in 
 J_{k,\mathcal{M}}^{(n')\, cusp}$ and any
 $\psi_2 \in 
 J_{k,\mathcal{M}\left[\left(\begin{smallmatrix} N & 0 \\ 0 & 1 \end{smallmatrix}
  \right) \right]}^{(n')\, cusp}$.
\end{lemma}
\begin{proof}
The proof is similar to the one of~Lemma \ref{lem:adjoint_U}.
\end{proof}

For the sake of simplicity we set
\begin{eqnarray*}
 V_{\alpha,n'-\alpha}^*(p^2) 
 &:=&
 V^*(\mbox{diag}(1_{\alpha},p 1_{n'-\alpha}, p^2 1_{\alpha}, p 1_{n'-\alpha}))
\end{eqnarray*}
for any prime $p$ and for any $\alpha$ $(0\leq \alpha \leq n')$.

\subsection{Adjoint map $\tilde{V}^*$ for Jacobi forms of half-integral weight}
We define the index-shift map $\tilde{V}^*_{\alpha,n'-\alpha}(p^2)$ for Jacobi forms
of half-integral weight.

\begin{df}\label{df:V*_pp}
For any prime $p$ and for any $\phi \in J_{k-\frac12,mp^2}^{+(n')}$
we define
\begin{eqnarray*}
  \phi|\tilde{V}^*_{\alpha,n'-\alpha}(p^2)
  &:=&
  p^{k(2n'+1)-n'(n'+\frac72)+\frac12 \alpha}
  \iota_{\M}\left( \left(\iota_{\M\left[ \smat{p}{0}{0}{1} \right]}^{-1}(\phi)\right)|V^*_{\alpha,n'-\alpha}(p^2)\right) .
\end{eqnarray*}
\end{df}

\begin{lemma}\label{lem:v_tilde_adj}
For any prime $p$,
for any $\phi_1 \in J_{k-\frac12,m}^{+(n')\, cusp}$
and for any $\phi_2 \in J_{k-\frac12,mp^2}^{+(n')\, cusp}$,
we have
\begin{eqnarray*}
  \langle \phi_1|\tilde{V}_{\alpha,n'-\alpha}(p^2), \phi_2 \rangle
  &=&
  \langle \phi_1, \phi_2|\tilde{V}^*_{\alpha,n'-\alpha}(p^2) \rangle.
\end{eqnarray*}
\end{lemma}
\begin{proof}
This is due to Lemma~\ref{lem:iota_petersson}, Lemma~\ref{lem:iota_v} and Lemma~\ref{lem:adjoint_V}.
\end{proof}

\begin{lemma}\label{lem:v_tilde_adj_2}
 We have
 \begin{eqnarray*}
   \phi| \tilde{V}^*_{\alpha,n'-\alpha}(p^2)
   &=&
   \phi | \tilde{V}_{\alpha,n'-\alpha}(p^2) | U^*_{p^2}
 \end{eqnarray*}
 for any prime $p$ and for any $\phi \in J_{k-\frac12,mp^2}^{+(n')\, cusp}$.
\end{lemma}
\begin{proof}
This lemma follows from Lemma~\ref{lem:V*_VU*}, Definition~\ref{df:U*_N},
Lemma~\ref{lem:iota_v} and Definition~\ref{df:V*_pp}.
\end{proof}

\subsection{Adjoint map $D^*$ for Jacobi forms of half-integral weight}
Let $n' = 2n-2$ and let the symbol $D_{n'}(N^2)$ be as in \S\ref{s:maass_relation}.

In this section we will give the adjoint map of $D_{n'}(N^2)$ with respect to the
Petersson inner product.

\begin{prop}\label{prop:adjoint_D}
 The adjoint map $D_{n'}^*(N^2)$ of $D_{n'}(N^2)$ with respect to
 the Petersson inner product is given by
 \begin{eqnarray*}
   \phi | D_{n'}^*(N^2) &=& \phi|D_{n'}(N^2)|U_{N^2}^*.
 \end{eqnarray*}
 for any $\phi \in J_{k-\frac12,m N^2}^{+(n')}$.
 It means we have
 \begin{eqnarray*}
   \langle \psi_1|D_{n'}(N^2), \psi_2 \rangle
   &=&
   \langle \psi_1, \psi_2|D_{n'}(N^2)|U_{N^2}^* \rangle
 \end{eqnarray*}
  for any $\psi_1 \in J_{k-\frac12,m}^{+(n')\, cusp}$ and any
 $\psi_2 \in J_{k-\frac12,mN^2}^{+(n')\, cusp}$.
\end{prop}
\begin{proof}
For a prime $p$
we write $V_{n'}(p^2) = \tilde{V}_{1,2n-3}(p^2)$ and $V_{n'}^*(p^2) = \tilde{V}_{1,2n-3}^*(p^2)$
for the sake of simplicity.
Due to 
Lemma~\ref{lem:v_tilde_adj}, Lemma~\ref{lem:v_tilde_adj_2},
Lemma~\ref{lem:u_adj} and Lemma~\ref{lem:u_adj_2},
for Jacobi cusp forms $\psi_1$ and $\psi_2$ with suitable indices, we have
\begin{eqnarray*}
   \langle \psi_1| U_{p^{2\delta}} | V_{n'}(p^2), \psi_2 \rangle  
   &=&
   \langle \psi_1| U_{p^{2\delta}} , \psi_2 |V_{n'}^*(p^2) \rangle  
   \ = \
   \langle \psi_1| U_{p^{2\delta}} , \psi_2 |V_{n'}(p^2) | U_{p^2}^* \rangle  \\   
   &=&
   \langle \psi_1 , \psi_2 |  V_{n'}(p^2)  | U_{p^{2\delta}} | U_{p^{4\delta+2}}^*  \rangle 
   \ = \
   \langle \psi_1 , \psi_2 |   U_{p^{2\delta}} | V_{n'}(p^2)  | U_{p^{4\delta+2}}^*  \rangle .
\end{eqnarray*}
Since $D_{n'}(N)$ is generated by a linear combination
of compositions of $V_{n'}(p^2)$ and $U_{p^2}$,
and since $\{V_{n'}(p^2)\}_p$ and $\{U_N\}_N$ are compatible,
we conclude 
  $\langle \psi_1|D_{n'}(N^2), \psi_2 \rangle
   =
   \langle \psi_1, \psi_2|D_{n'}(N^2)|U_{N^2}^* \rangle$
by similar arguments as above.
\end{proof}

\section{Jacobi-Eisenstein series and index-shift maps}\label{s:Jacobi_Eisenstein}
Let $n'$ be a natural number.
In this section we introduce Fourier-Jacobi coefficients $\{e_{k-\frac12,m}^{(n')}\}_m$
of a generalized Cohen-Eisenstein series.
Here $e_{k-\frac12,m}^{(n')} \in J_{k-\frac12,m}^{+(n')}$.
Moreover, we shall give a formula for $e_{k-\frac12,mN^{2}}^{(n')}| U_{N}^*$,
where $U^*_{N}$ is the adjoint map introduced in \S\ref{ss:adjoint_map_of_u_half}.

\subsection{Jacobi-Eisenstein series}
In this subsection we review some Jacobi-Eisenstein series and their Fourier-Jacobi coefficients.

Let $r > 0$ be an integer and let $\M \in L_r^+$.
For an even integer $k > n' + r + 1$, the Jacobi-Eisenstein series of weight $k$ 
of index $\M$ of degree $n'$ is defined by
\begin{eqnarray*}
  E_{k,\M}^{(n')}
  &:=&
  \sum_{M \in \Gamma_{\infty}^{(n')}\backslash \Gamma_n}\sum_{\lambda \in \Z^{(n',r)}}
  1|_{k,\mathcal{M}}([(\lambda,0),0_r], M) .
\end{eqnarray*}

We take $\M = 1 \in \Z_{>0}$. We define
\begin{eqnarray*}
  \mathcal{H}_{k-\frac12}^{(n')} &:=& \iota_{1} (E_{k,1}^{(n')}),
\end{eqnarray*}
where the linear map $\iota_1  :  J_{k,1}^{(n')} \rightarrow M_{k-\frac12}^{+(n')}$ 
is defined in \S\ref{s:lin_iso_jacobi}.
The form $\mathcal{H}_{k-\frac12}^{(n')}$ belongs to the plus-space $M_{k-\frac12}^{+(n')}$.
In this article
we call $\mathcal{H}_{k-\frac12}^{(n')}$ a generalized Cohen-Eisenstein series,
since $\mathcal{H}_{k-\frac12}^{(1)}$ has been introduced by Cohen \cite{Co}.
(See also \cite{Ar98}).

For $m \in \Z$ we  denote by $e_{k-\frac12,m}^{(n')}$ the $m$-th Fourier-Jacobi
coefficient of $\mathcal{H}_{k-\frac12}^{(n'+1)}$, it means 
\begin{eqnarray*}
  \mathcal{H}_{k-\frac12}^{(n'+1)}\left(\begin{pmatrix} \tau & z \\ ^t z & \omega \end{pmatrix} \right)
  &=&
  \sum_{m \in \Z} e_{k-\frac12,m}^{(n')}(\tau,z) e(m \omega),
\end{eqnarray*}
for $\tau \in \H_{n'}$, $z \in \C^{n'}$ and $\omega \in \H_1$.
We remark that $e_{k-\frac12,m}^{(n')} = 0$ unless $m \equiv 0$, $3$ $\!\! \mod 4$,
since $\mathcal{H}_{k-\frac12}^{(n'+1)}$ belongs to the plus-space $M_{k-\frac12}^{+(n'+1)}$.
Remark also that $e_{k-\frac12,m}^{(n')}$ is a Jacobi form which belongs to $J_{k-\frac12,m}^{+(n')}$.

On the other hand,
we take a Fourier-Jacobi expansion of $E_{k,1}^{(n'+1)}$:
\begin{eqnarray*}
  E_{k,1}^{(n'+1)}\left(\begin{pmatrix} \tau & z_1 \\ ^t z_1 & \omega_1 \end{pmatrix},
  \begin{pmatrix} z_2 \\ \omega_2 \end{pmatrix} \right)
  e(\omega_3)
  &=&
  \sum_{\begin{smallmatrix} \M \in L_2^* \\ 
  \M = \left( \begin{smallmatrix} * & * \\ * & 1 \end{smallmatrix}\right) \end{smallmatrix}} 
  e_{k,\M}^{(n')}(\tau, (z_1,z_2)) 
  e(\M \begin{pmatrix} \omega_1 & \omega_2 \\
  ^t \omega_2 & \omega_3 \end{pmatrix}),
\end{eqnarray*}
where $\begin{pmatrix} \tau & z_1 & z_2 \\ ^t z_1 & \omega_1 & \omega_2 \\
^t z_2 & ^t \omega_2 & \omega_3 \end{pmatrix} \in \H_{n'+2}$,
$\tau \in \H_{n'}$, $\omega_1 \in \H_1$ and $\omega_3 \in \H_1$.
The form $e_{k,\M}^{(n')}$ is a Jacobi form which belongs to $J_{k,\M}^{(n')}$.
\begin{lemma}\label{lem:iota_jacobi_eisen}
 For $\M = \begin{pmatrix} * & * \\ * & 1 \end{pmatrix} \in L_2^+$ we put $m = \det(2 \M)$.
 Then we have
 \begin{eqnarray*}
    \iota_{\M}(e_{k,\M}^{(n')}) &=& e_{k-\frac12,m}^{(n')}.
 \end{eqnarray*}
\end{lemma}
\begin{proof}
 From the definition of two linear maps  $\iota_1$ and $\iota_{\M}$
 and from the definition of the Fourier-Jacobi expansions,
 the diagram
 $$
\begin{CD}
 J_{k,1}^{(n'+1)} @>\iota_1 >> M_{k-\frac12}^{+(n'+1)} \\
 @VVV
 @VVV \\
 J_{k,\mathcal{M}}^{(n')} @>\iota_{\mathcal{M}}
 >> J_{k-\frac12,m}^{+(n')} 
\end{CD}
$$
is commutative, where two down arrows are given by the Fourier-Jacobi expansions.
Thus this lemma follows from the definitions of $e_{k-\frac12,m}^{(n')}$ and 
$e_{k,\M}^{(n')}$.
\end{proof}

We now describe $e_{k,\M}^{(n')}$ as a linear combination of the Jacobi-Eisenstein
series $\{ E_{k,\M'}^{(n')} \}_{\M'}$.

We denote by $h_{k-\frac12}(m)$ the $m$-th Fourier coefficient of 
the Cohen-Eisenstein series $\mathcal{H}_{k-\frac12}^{(1)}$.
It means $\displaystyle{\mathcal{H}_{k-\frac12}^{(1)}(\tau) = \sum_{m} h_{k-\frac12}(m) e(m\tau)}$.

Let $m$ be a natural number such that $-m = D_0 f^2$ with a fundamental discriminant
$D_0$ and with a natural number $f$.
It is obvious that $m \equiv 0, 3 \!\! \mod 4$.
We define
\begin{eqnarray*}
  g_k(m) &:=& \sum_{d|f} \mu(d)\, h_{k-\frac12}\!\left( \frac{m}{d^2} \right),
\end{eqnarray*}
where $\mu$ is the M\"obius function.

We will use the following lemma for the proof of Proposition~\ref{prop:e_u_tilde}.
\begin{lemma}\label{lem:gk}
Let $m'$ be a natural number such that $-m' \equiv 0$, $1 \mod 4$. 
Then for any prime $p$ we have
\begin{eqnarray*}
 g_k(p^2m') &=& 
  \left(p^{2k-3} - \left(\frac{-m'}{p} \right) p^{k-2} \right) g_k(m') . 
\end{eqnarray*}
\end{lemma}
\begin{proof}
See~\cite[Lemma 3.2]{CE_N}.
\end{proof}

\begin{lemma}\label{lem:e_E}
  For $\M = \begin{pmatrix} * & * \\ * & 1  \end{pmatrix} \in L_2^+$
  we set $m = \det(2\M)$.
  Let $D_0$ and $f$ be as above.
  If $k > n' + 3$, then
  \begin{eqnarray*}
 e_{k,\mathcal{M}}^{(n')}(\tau,z)
 &=&
 \sum_{d|f} g_k\!\left(\frac{m}{d^2}\right) E_{k,\mathcal{M}[{W_d}^{-1}]}^{(n')}(\tau,z {^t W_d}) ,
\end{eqnarray*}
where we chose a matrix $W_d \in \Z^{(2,2)}$ for each $d$ which satisfies the conditions
$\det(W_d) = d$,
$\mathcal{M} \left[ {W_d}^{-1} \right] \in  L_2^+$
and $\mathcal{M} \left[ { W_d}^{-1} \right] = \begin{pmatrix} * & * \\ * & 1 \end{pmatrix}$.
If $W_d$ satisfies these conditions, the right hand side of the identity does not depend
on the choice of $W_d$.
\end{lemma}
\begin{proof}
The reader is referred to~\cite[Proposition 3.3]{CE_N}.
This formula has originally been given in~\cite[Satz 7]{Bo}.
\end{proof}

\begin{df}
  For $\M = \begin{pmatrix} * & * \\ * & 1  \end{pmatrix} \in L_2^+$ we set
  $m = \det(2\M)$.
  We define
  \begin{eqnarray*}
    E_{k-\frac12,m}^{(n')} &:=& \iota_{\M}(E_{k,\M}^{(n')}).
  \end{eqnarray*}
\end{df}

\begin{lemma}\label{lem:e_E_half}
Let the symbols $\M$, $m$ and $f$ be as above.
Then
we have
$E_{k-\frac12,m}^{(n')} \in J_{k-\frac12,m}^{+(n')}$ and
\begin{eqnarray*}
  e_{k-\frac12,m}^{(n')}(\tau,z)
  &=&
   \sum_{d|f} g_k\!\left(\frac{m}{d^2}\right) E_{k-\frac12,\frac{m}{d^2}}^{(n')}(\tau,dz).
\end{eqnarray*}
\end{lemma}
\begin{proof}
 This lemma follows from Proposition~\ref{prop:matrix_integer_isom},
 Lemma~\ref{lem:iota_jacobi_eisen}
 and Lemma~\ref{lem:e_E}.
\end{proof}

\begin{lemma}\label{lem:E_U_N}
Let the symbols be as above.
We have
\begin{eqnarray*}
  E_{k-\frac12,mN^2}^{(n')}|U_{N}^{*}
  &=&
  N^{-n'} E_{k-\frac12,m}^{(n')}
\end{eqnarray*}
for any natural number $N$,
where $U^*_N$ is the adjoint map of $U_N$ introduced in \S\ref{ss:adjoint_map_of_u_half}. 
\end{lemma}
\begin{proof}
Since
\begin{eqnarray*}
  E_{k-\frac12,mN^2}^{(n')}|U_{N}^{*}
  &=&
  \iota_{\M}\left( \iota_{\M[\smat{N}{0}{0}{1} ]}^{-1}(E_{k-\frac12,mN^2}^{(n')} ) 
  | U^*_{[\smat{N}{0}{0}{1}]}
   \right) \\
  &=&
  \iota_{\M}\left( E_{k,\M[\smat{N}{0}{0}{1} ]}^{(n')} | U^*_{[\smat{N}{0}{0}{1}]} \right),
\end{eqnarray*}
it is enough to show
\begin{eqnarray*}
  E_{k,\M[\smat{N}{0}{0}{1} ]}^{(n')} | U^*_{[\smat{N}{0}{0}{1}]}
  &=&
  N^{-n'} E_{k,\M}^{(n')}.
\end{eqnarray*}
From the definition we have
\begin{eqnarray*}
  E_{k,\M[\smat{N}{0}{0}{1} ]}^{(n')}(\tau,z)
  &=&
  \sum_{M \in \Gamma_{\infty}^{(n')}} \sum_{\lambda \in \Z^{(n',2)}}
  J_{k,\M\smat{N}{0}{0}{1} }\left( 
  \left( [(\lambda, 0),0_2], M \right)
  , (\tau,z )\right)^{-1} \\
  &=&
  \sum_{M \in \Gamma_{\infty}^{(n')}} \sum_{\lambda \in \Z^{(n',2)}}
  J_{k,\M}\left( 
  \left( [(\lambda \smat{N}{0}{0}{1} , 0),0_2], M \right)
  , (\tau,z \smat{N}{0}{0}{1} )\right)^{-1}.
\end{eqnarray*}
We write $X_{\lambda',\mu'} = [((\lambda',0),(\mu',0)),0_2]$ with $\lambda'$, $\mu'$ $\in \Z^{(n',1)}$.
Then
\begin{eqnarray*}
  \left(E_{k,\M[\smat{N}{0}{0}{1} ]}^{(n')}|U_{\smat{N^{-1}}{0}{0}{1} } |_{k,\M}X_{\lambda',\mu'} \right)(\tau,z) 
  &=&
  \sum_{M \in \Gamma_{\infty}^{(n')}} \sum_{\lambda \in \Z^{(n',2)}}
  J_{k,\M}\left( 
  \gamma
  , (\tau,z )\right)^{-1},
\end{eqnarray*}
where we write
$\gamma = \left( [(\lambda \smat{N}{0}{0}{1}+(\lambda'',0) , (\mu'',0)),0_2], M \right)$
and where
\begin{eqnarray*}
  \begin{pmatrix}   \lambda'' \\ \mu'' \end{pmatrix}
  &=&
  ^t {M ^{-1}}
  \begin{pmatrix} \lambda' \\ \mu''  \end{pmatrix}.
\end{eqnarray*}
Thus we obtain
\begin{eqnarray*}
  E_{k,\M[\smat{N}{0}{0}{1} ]}^{(n')} | U^*_{[\smat{N}{0}{0}{1}]}
  &=&
  N^{-2n'}
  \sum_{\lambda', \mu' \in (\Z/ N\Z)^{(n',1)}}
    \left(E_{k,\M[\smat{N}{0}{0}{1} ]}^{(n')}|U_{\smat{N^{-1}}{0}{0}{1} } |_{k,\M}X_{\lambda',\mu'} \right) \\
  &=&
  N^{-n'} E_{k,\M}^{(n')}.
\end{eqnarray*}
\end{proof}

\begin{df}
  Let $N$ be a natural number and let $\nu = \mbox{ord}_p N$ be the largest integer such that
  $p^{\nu} | N$.
  We define
  \begin{eqnarray*}
  \Psi_p^{(n')}(N,X)
  &:=&
  \Psi_{p,D_0}^{(n')}(N,X)
  \ = \  
  \frac{X^{\nu+1} - X^{-(\nu+1)}}{X - X^{-1}}
  -
  \left( \frac{D_0}{p}  \right) p^{-\frac{n'}{2} - \frac12}
  \frac{X^{\nu} - X^{-\nu}}{X - X^{-1}}.
\end{eqnarray*}
\end{df}

\begin{prop}\label{prop:e_u_tilde}
Let the symbols $m$, $D_0$ and $f$ be as in Lemma~\ref{lem:e_E}.
Recall $D_0$ is the fundamental discriminant such that $f = \sqrt{m/|D_0|}$ is a natural number.
If natural number $N$ is coprime to $f$, then we obtain
\begin{eqnarray*}
    e_{k-\frac12,mN^2}^{(n')}|U^*_{N}
  &=&
  N^{(k-\frac{n'}{2}-\frac32)}
  \prod_{p|N}
    \Psi_p^{(n')}(p^{\nu},p^{k-\frac{n'}{2}-\frac32})
  e_{k-\frac12,m}^{(n')},
\end{eqnarray*}
where $\nu = \mbox{ord}_p N$.
\end{prop}
\begin{proof}
We assume that $f$ and $p$ are coprime.
Due to Lemma~\ref{lem:e_E_half}, Lemma~\ref{lem:gk}, Lemma~\ref{lem:u_adj_2}
and Lemma~\ref{lem:E_U_N}, we have
\begin{eqnarray*}
 e_{k-\frac12,mp^{2\nu}}^{(n')} | U_{p^{\nu}}^{*}
 &=&
 \sum_{d|f} \sum_{i=0}^{\nu} g_k\left(\frac{m}{d^2} p^{2(\nu - i )} \right) E_{k-\frac12,\frac{m}{d^2}p^{2(\nu - i)}}^{(n')} |U_{d p^i} | U_{p^{\nu}}^{*}\\
 &=&
 \sum_{d|f} g_k\left(\frac{m}{d^2} \right) \\
 &&
 \left\{ \sum_{i=0}^{\nu-1} 
 \left(p^{(2k-3)(\nu-i)}
   - \left(\frac{D_0}{p}\right)p^{k-2 + (2k-3)(\nu-i-1)}\right)
   E_{k-\frac12,\frac{m}{d^2}p^{2(\nu - i)}}^{(n')} |U_{d} |  U_{p^{\nu-i}}^{*}
 \right. \\
 &&
 \left.
  +  E_{k-\frac12,\frac{m}{d^2}}^{(n')} |U_{d} 
 \right\}\\
 &=&
 \sum_{d|f} g_k\left(\frac{m}{d^2} \right)  E_{k-\frac12,\frac{m}{d^2}}^{(n')} |U_{d} \\
 &&
 \left\{
 1  + \sum_{i=0}^{\nu-1} 
 \left(p^{(2k-3)(\nu-i)}
   - \left(\frac{D_0}{p}\right)p^{k-2 + (2k-3)(\nu-i-1)}\right)
  p^{ - n'(\nu - i)}
 \right\}\\
 &=&
 e_{k-\frac12, m}^{(n')} 
 \left\{ p^{(k-\frac{n'}{2}-\frac32)\nu}
 \frac{p^{(k-\frac{n'}{2}-\frac32)(\nu+1)} - p^{-(k-\frac{n'}{2}-\frac32)(\nu+1)}}{
 p^{(k-\frac{n'}{2}-\frac32)} - p^{-(k-\frac{n'}{2}-\frac32)}} \right.\\
 &&
 \left.
 - \left( \frac{D_0}{p}  \right) p^{k-2-n'+(k-\frac{n'}{2}-\frac32)(\nu-1)}
 \frac{p^{(k-\frac{n'}{2}-\frac32)\nu} - p^{-(k-\frac{n'}{2}-\frac32)\nu}}{
 p^{(k-\frac{n'}{2}-\frac32)} - p^{-(k-\frac{n'}{2}-\frac32)}} 
 \right\}
 .
 \end{eqnarray*}
 Thus we have
 \begin{eqnarray*}
  e_{k-\frac12,mp^{2\nu}}^{(n')} | U_{p^{\nu}}^{*}
  &=&
  p^{(k-\frac{n'}{2}-\frac32)\nu}
  \Psi_p^{(n')}(p^{\nu},p^{k-\frac{n'}{2}-\frac32})
  e_{k-\frac12,m}^{(n')}.
\end{eqnarray*}
By induction with respect to $m$ and $p$, we obtain this proposition.
\end{proof}

\section{Proof of main theorem}\label{s:proof_of_main_theorem}
In this section we shall prove Theorem~\ref{th:main_dshalf}.
Let the symbols $G$ and $\phi_m$ be as in \S\ref{s:intro} and \S\ref{s:Ikeda_lift}.
The adjoint map $U_N^*$ was introduced in Definition~\ref{df:U*_N}.
The index-shift map $D_{2n-2}(N^2)$ was denoted in \S\ref{s:maass_relation}.
The adjoint map $\tilde{D}_{2n-2}(N^2)^*$ of $D_{2n-2}(N^2)$ was obtained
in Proposition~\ref{prop:adjoint_D}.

We write $-m = D_0 f^2$ with a fundamental discriminant $D_0$ and with
a natural number $f$.

\begin{lemma}\label{lem:phi_U*}
We assume that $f$ and $N$ are coprime. Then we have
\begin{eqnarray*}
  \phi_{mN^2}|U^*_{N}
  &=&
  N^{k-n-\frac12} \prod_{p|N} \Psi_p^{(2n-2)}(p^{\nu}, \alpha_p) \phi_m.
\end{eqnarray*}
\end{lemma}
\begin{proof}
Due to Proposition~\ref{prop:e_u_tilde} we have
\begin{eqnarray*}
    e_{k'-\frac12,mN^2}^{(2n-2)}|U^*_{N}
  &=&
  N^{k'-n-\frac12}
  \prod_{p|N}
    \Psi_p^{(2n-2)}(p^{\nu},p^{k'-n-\frac12})
  e_{k'-\frac12,m}^{(2n-2)}
\end{eqnarray*}
for infinitely many natural number $k'$.
By a standard argument for Ikeda lifts, we obtain this lemma.
\end{proof}

\begin{lemma}\label{lem:phi_m_DD}
We assume that $f$ and $N$ are coprime. Then we have
\begin{eqnarray*}
  &&
   \phi_m|D_{2n-2}(N^2)|D_{2n-2}^*(N^2)  \\
   &=&
     \sum_{d | N} \left(Nd^{-1}\right)^{2k-2n-1}  d^{2k-3}
  \sum_{d_1 | \frac{N}{d}} \mu(d_1) \left( \frac{D_0}{d_1} \right)
  d_1^{n-\frac32}
  \prod_{p | \frac{N^2}{d^2 d_1}} \Psi_p^{(2n-2)}(p^{\nu},\alpha_p)
  \phi_{m}.
\end{eqnarray*}
\end{lemma}
\begin{proof}
By virtue of Proposition~\ref{prop:adjoint_D}, Lemma~\ref{lem:DD}(4),
Proposition~\ref{prop:maass_relation_D},
Lemma~\ref{lem:u_adj_2}(2)
and Lemma~\ref{lem:phi_U*},
we obtain
\begin{eqnarray*}
  &&
  \phi_m|D_{2n-2}(N^2)|D_{2n-2}^*(N^2) \\
  &=&
   \phi_m|D_{2n-2}(N^2) | D_{2n-2}(N^2) | U^*_{N^2} \\
  &=&
  \phi_m| \sum_{d | N} d^{2k-3}
  \sum_{d_1 | \frac{N}{d}} \mu(d_1) \left( \frac{D_0}{d_1} \right)
  d_1^{k-2}
  U_{d^2 d_1}
  | D_{2n-2}\left(\frac{N^4}{d^4 d_1^2} \right)
   | U^*_{N^2}\\
  &=&
   \sum_{d | N} d^{2k-3}
  \sum_{d_1 | \frac{N}{d}} \mu(d_1) \left( \frac{D_0}{d_1} \right)
  d_1^{k-2}
  \phi_{\frac{mN^4}{d^4d_1^2}} |U_{d^2 d_1} |  U^*_{N^2}\\
  &=&
   \sum_{d | N} d^{2k-3}
  \sum_{d_1 | \frac{N}{d}} \mu(d_1) \left( \frac{D_0}{d_1} \right)
  d_1^{k-2}
  \phi_{\frac{mN^4}{d^4d_1^2}} | U^*_{\frac{N^2}{d^2 d_1}}\\
  &=&
   \sum_{d | N} d^{2k-3}
  \sum_{d_1 | \frac{N}{d}} \mu(d_1) \left( \frac{D_0}{d_1} \right)
  d_1^{k-2}
  \left( N^2 d^{-2} d_1^{-1} \right)^{k-n-\frac12}
  \prod_{p | \frac{N^2}{d^2 d_1}} \Psi_p^{(2n-2)}(p^{\nu},\alpha_p)
  \phi_{m} \\
  &=&
  \sum_{d | N} \left(Nd^{-1}\right)^{2k-2n-1}  d^{2k-3}
  \sum_{d_1 | \frac{N}{d}} \mu(d_1) \left( \frac{D_0}{d_1} \right)
  d_1^{n-\frac32}
  \prod_{p | \frac{N^2}{d^2 d_1}} \Psi_p^{(2n-2)}(p^{\nu},\alpha_p)
  \phi_{m} .
\end{eqnarray*}
\end{proof}

We shall now prove Theorem~\ref{th:main_dshalf}.
We have
\begin{eqnarray*}
     \sum_{\begin{smallmatrix} m \in \Z_{>0} \\ -m \equiv 0,1 \!\! \mod 4 \end{smallmatrix}}
     \frac{\langle \phi_m, \phi_{m}\rangle}{m^{s+k-n-\frac12}} 
    &=&
    \sum_{D_0} \frac{1}{|D_0|^{s+k-n-\frac12}}
      \sum_{N = 1}^{\infty} \frac{\langle \phi_{|D_0| N^2}, \phi_{|D_0|N^2}\rangle}{N^{2(s+k-n-\frac12)}},
\end{eqnarray*}
where $D_0$ runs over all negative fundamental discriminants.
Due to Proposition~\ref{prop:maass_relation_D},
we have $\phi_{|D_0| N^2} = \phi_{|D_0|}|D_{2n-2}(N^2)$
for any natural number $N$.
Here $D_{2n-2}(N^2)$ is defined for a fixed $|D_0|$.
Thus,  by virtue of Lemma~\ref{lem:phi_m_DD}, we have
\begin{eqnarray*}
  &&
  \sum_{N = 1}^{\infty} \frac{\langle \phi_{|D_0| N^2}, \phi_{|D_0|N^2}\rangle}{N^{2s}} \\
  &=&
  \sum_{N = 1}^{\infty} \frac{\langle \phi_{|D_0|}|D_{2n-2}(N^2), \phi_{|D_0|}|D_{2n-2}(N^2)\rangle}{N^{2s}} \\
  &=&
  \sum_{N = 1}^{\infty} \frac{\langle \phi_{|D_0|}|D_{2n-2}(N^2) D_{2n-2}^*(N^2), \phi_{|D_0|}\rangle}{N^{2s}} \\
  &=&
  \langle \phi_{|D_0|}, \phi_{|D_0|}\rangle \sum_{N = 1}^{\infty}
    \sum_{d | N} \left(\frac{N}{d}\right)^{-2s + 2k-2n-1}  d^{-2s + 2k - 3}  \\
  &&
  \times
  \sum_{d_1 | \frac{N}{d}} \mu(d_1) \left( \frac{D_0}{d_1} \right)
  d_1^{n-\frac32}
  \prod_{p | \frac{N^2}{d^2 d_1}} \Psi_p^{(2n-2)}(p^{\nu},\alpha_p)\\
  &=&
  \langle \phi_{|D_0|}, \phi_{|D_0|}\rangle \zeta(2s-2k+3)\\
  &&
  \times \sum_{N = 1}^{\infty}
    N^{-2s + 2k-2n-1}
  \sum_{d_1 | N} \mu(d_1) \left( \frac{D_0}{d_1} \right)
  d_1^{n-\frac32}
  \prod_{p | \frac{N^2}{d_1}} \Psi_p^{(2n-2)}(p^{\nu},\alpha_p)\\
  &=&
  \langle \phi_{|D_0|}, \phi_{|D_0|}\rangle \zeta(2s-2k+3)\\
  &&
  \times
  \prod_{p} \left\{
    1 +
    \sum_{\delta = 1}^{\infty}
    p^{\delta(-2s + 2k-2n-1)}
    \left(
     \Psi_p^{(2n-2)}(p^{2\delta},\alpha_p)
     - \left( \frac{D_0}{p} \right) p^{n-\frac32} \Psi_p^{(2n-2)}(p^{2\delta -1}, \alpha_p)
    \right)
    \right\} \\
    &=&
    \langle \phi_{|D_0|}, \phi_{|D_0|}\rangle \,  \zeta(2s-2k+3)\\
    &&
    \times
    \prod_p
    \left\{ (1 - \alpha_p^2 p^{-2s + 2k - 2n - 1}) (1 - \alpha_p^{-2} p^{-2s + 2k - 2n - 1}) \right\}^{-1} \\
    &&
    \times
    \left\{
      1 + p^{-2s + 2k - 2n - 1} - \left(\frac{D_0}{p} \right) p^{-2s + 2k - 3n - \frac12}(\alpha_p + \alpha_p^{-1})
    \right. \\
    &&
    \left.
    + \left( \frac{D_0}{p} \right)^2 p^{-2s + 2k - 2n - 2} ( 1 + p^{-2s + 2k - 2n - 1})
    - \left( \frac{D_0}{p} \right) p^{-2s + 2k - n - \frac52} (\alpha_p + \alpha_p^{-1})
    \right\} \\
    &=&
    \langle \phi_{|D_0|}, \phi_{|D_0|}\rangle \, \zeta(2s-2k+3) \, L(2s-2k+2n+1, f, Ad) \\
    &&
    \times
    \prod_p
    \left\{
      \left(1 - p^{-4s+4k-4n-2} \right) \left( 1 + \left( \frac{D_0}{p} \right)^2 p^{-2s + 2k - 2n - 2}\right)
    \right. \\
    &&
    \left.
    \qquad
     - \left( \frac{D_0}{p} \right) \left( 1 - p^{-2s+2k-2n-1} \right)
     \left( 1 + p^{2n-2} \right)
     p^{-2s+2k-3n-\frac12}(\alpha_p + \alpha_p^{-1})
     \right\}\\
     &=&
    \langle \phi_{|D_0|}, \phi_{|D_0|}\rangle \, \zeta(2s-2k+3)
    \, \zeta(4s - 4k + 4n + 2)^{-1}  
     L(2s-2k+2n+1, f, Ad) \\
    &&
    \times
    \prod_{p \nmid D_0}
    \left\{
       1 + p^{-2s + 2k - 2n - 2}
     - \frac{\left( \frac{D_0}{p} \right)
     \left( 1 + p^{2n-2} \right)
     p^{-k+1} a_f(p) }{
      p^{2s-2k+2n+1} + 1
     }
     \right\}.     
\end{eqnarray*}
Therefore we conclude Theorem~\ref{th:main_dshalf}.


\vspace{1cm}

\noindent
Department of Mathematics, Joetsu University of Education,\\
1 Yamayashikimachi, Joetsu, Niigata 943-8512, JAPAN\\
e-mail hayasida@juen.ac.jp


\begin{thebibliography}{99}
 \bibitem[Ar 98]{Ar98}
  T.~Arakawa :
  K\"ocher-Maass Dirichlet series corresponding to Jacobi forms and Cohen Eisenstein series,
  {\it Comment. Math. Univ. St. Paul.} {\bf 47} (1998), 93--122.
 \bibitem[Bo 83]{Bo}
   S.~B\"ocherer : 
     \"Uber die Fourier-Jacobi-Entwicklung Siegelscher Eisensteinreihen, 
     {\it Math.\ Z.}\ {\bf 183} (1983), 21--46
 \bibitem[Co 75]{Co}
   H.~Cohen : Sums involving the values at negative integers of $L$-functions of
   quadratic characters, {\it Math.\ Ann.}\ {\bf 217} (1975), 171--185.
 \bibitem[Du 95]{Du}
  J.~Dulinski : A decomposition theorem for Jacobi forms,
  {\it Math. Ann.} {\bf 303} (1995), no. 3, 473--498.
 \bibitem[E-Z 85]{EZ}
  M.~Eichler and D.~Zagier : Theory of Jacobi Forms, Progress in
	 Math.\ {\bf 55}, Birkh\"auser, Boston-Basel-Stuttgart, (1985).
\bibitem[H 11]{FJlift}
 S.~Hayashida~: Fourier-Jacobi expansion and the Ikeda lift,
    {\it Abh. Math. Semin. Univ. Hambg.} {\bf 81} (2011), 1--17.
\bibitem[H 16]{CE_N}
 S.~Hayashida~:
  Lifting from two elliptic modular forms to Siegel modular forms of half-integral weight of even degree. 
  {\it Doc. Math.} {\bf 21} (2016), 125--196.
\bibitem[H1 18]{matrix_integer}
 S.~Hayashida~:
   On Kohnen plus-space of Jacobi forms of half integral weight of matrix index,
   to be appeared in Osaka Journal of Mathematics.
\bibitem[H2 18]{RSJacobi}
    S.~Hayashida :  Rankin-Selberg method for Jacobi forms and
 for plus space of Jacobi forms in higher degrees, preprint. 
  arXiv:1804.04319 
\bibitem[Ib 92]{Ib}
  T.~Ibukiyama : On Jacobi forms and Siegel modular forms of half
	 integral weights, {\it Comment.\ Math.\ Univ.\ St.\ Paul.}\ {\bf 41} 
	 No.2 (1992), 109--124.
\bibitem[I-K 03]{IbKa}
  T.~Ibukiyama and H.~Katsurada :
  An explicit formula for Koecher-Maa\ss\ Dirichlet series for Eisenstein series of Klingen type.
  {\it J. Number Theory}  {\bf 102} No. 2, (2003),  223--256.
\bibitem[Ik 01]{Ik}
  T.~Ikeda : On the lifting of elliptic cusp forms to Siegel cusp forms of degree 2n, 
         {\it Ann.\ of Math.\ (2)} {\bf 154} no.3, (2001), 641--681.
\bibitem[K-K 08]{KaKa:Dirichlet}
  H.~Katsurada and H.~Kawamura~:
  A certain Dirichlet series of Rankin-Selberg type associated with the Ikeda lifting,
  {\it J. Number Theory}\ {\bf 128}, no. 7 (2008), 2025--2052.
\bibitem[K-K15]{KaKa}
  H.~Katsurada and H.~Kawamura~:
  Ikeda's conjecture on the period of the Duke-Imamoglu-Ikeda lift.
     {\it Proc. Lond. Math. Soc. (3)}\, {\bf 111} No. 2, (2015),  445--483.
\bibitem[K-S 89]{KoSk}
  W.~Kohnen and N.-P. Skoruppa : 
  A certain Dirichlet series attached to Siegel modular forms of degree two.
  {\it Invent. Math.}\ {\bf 95}, no.3 (1989), 541--558. 
\bibitem[K-Z 81]{KoZa}
  W.~Kohnen and D.~Zagier : Values of L-series of modular forms at the center of the critical strip.
  {\it Invent. Math.} {\bf 64} No. 2, (1981),  175--198.
\bibitem[Kl 89]{Kl2}
  H.~Klingen :
  \"Uber Kernfunktionen f\"ur Jacobiformen und Siegelsche Modulformen,
  {\it Math. Ann.} {\bf 285} (1989), no. 3, 405--416.
 \bibitem[Zi 89]{Zi} 
  C.~Ziegler : Jacobi forms of higher degree, {\it Abh.\ Math.\ Sem.\ Univ.\  
  Hamburg.}\ {\bf 59} (1989), 191--224.  
\end{thebibliography}
\end{document}